\documentclass[12pt]{amsart}

\usepackage{amssymb,latexsym,eufrak,amsmath,amscd, graphicx}

\usepackage{amsfonts}
\usepackage{amscd}

\renewcommand{\phi}{\varphi}
\renewcommand{\epsilon}{\varepsilon}
\renewcommand{\theta}{\vartheta}

\newcommand\rounddown[1]{\left\lfloor#1\right\rfloor}
\newcommand\roundup[1]{\left\lceil#1\right\rceil}

\def\ZZ{{\mathbf Z}}
\def\NN{{\mathbf N}}
\def\FF{{\mathbf F}}
\def\CC{{\mathbf C}}
\def\AAA{{\mathbf A}}
\def\RR{{\mathbf R}}
\def\QQ{{\mathbf Q}}
\def\PP{{\mathbf P}}

\def\cJ{\mathcal{J}}

\def\cO{\mathcal{O}}

\def\fra{\mathfrak{a}}

\def\frb{\mathfrak{b}}

\def\frm{\mathfrak{m}}

\def\frq{\mathfrak{q}}

\def\*{{}^*\!}

 \DeclareMathOperator{\Spec}{Spec}
 
\DeclareMathOperator{\lct}{lct}
 
 \DeclareMathOperator{\ord}{ord}
 \DeclareMathOperator{\fpt}{fpt}

\DeclareMathOperator{\Arn}{Arn}

\def\.{\cdot}
\def\~{\widetilde}
\def\^{\widehat}

\def\o{\circ}

\newcommand{\llbracket}{[\negthinspace[}
\newcommand{\rrbracket}{]\negthinspace]}

\newtheorem{lemma}{Lemma}[section]
\newtheorem{theorem}[lemma]{Theorem}

\newtheorem{conjecture}[lemma]{Conjecture}

\theoremstyle{definition}
\newtheorem{definition}[lemma]{Definition}
\newtheorem{remark}[lemma]{Remark}

\newtheorem{example}[lemma]{Example}

\newtheorem{property}[lemma]{Property}

\theoremstyle{remark}
\newtheorem*{remark*}{Remark}
\newtheorem*{note*}{Note}

\begin{document}

\title{IMPANGA lecture notes on log canonical thresholds}
\author[M. Musta\c{t}\u{a}]{Mircea~Musta\c{t}\u{a}}

\address{Department of Mathematics, University of Michigan,
530 Church Street, 
Ann Arbor, MI 48109, USA}
\email{{\tt mmustata@umich.edu}}

\date{}

\maketitle

\begin{center}
\rm{Notes by Tomasz Szemberg}
\end{center}
\thispagestyle{empty}

\markboth{MIRCEA MUSTA\c{T}\u{A}}{AN INTRODUCTION TO LOG CANONICAL THRESHOLDS}


\section*{Introduction}

Let $H\subset \CC^n$ be a complex hypersurface defined by the polynomial 
$f\in\CC[x_1,\ldots,x_n]$. The problem of understanding the singularities of $H$ at a given point is classical. 
The topological study goes back to Milnor's
book \cite{Milnor}. In these notes, however, we will focus on an algebraic invariant, the log canonical threshold.

The two best-known invariants of the singularity of $f$ (or $H$) at a point $P\in H$ are the multiplicity
${\rm ord}_P(f)$ and the Milnor number $\mu_P(f)$ (in the case when $H$
has an isolated singularity at $P$).  They are both easy to define: 
${\rm ord}_P(f)$ is the smallest $|\alpha|$ with 
$\frac{\partial^{\alpha}f}{\partial x^{\alpha}}(P)\neq 0$, where
 $\alpha=(\alpha_1,\ldots,\alpha_n)\in\ZZ_{\geq 0}^n$ and
$|\alpha|=\sum_{i=1}^n\alpha_i$.
If $H$ is nonsingular in a punctured neighborhood of $P$, then
$$\mu_P(f)=\dim_{\CC}\cO_{\CC^n,P}/(\partial f/\partial x_1,\ldots,
\partial f/\partial x_n).$$

Note that both these invariants are integers. They both detect whether $P\in H$
is a singular point: this is the case if and only ${\rm ord}_P(f)\geq 2$, and
(assuming that $H$ is nonsingular in a punctured neighborhood of $P$)
if and only if $\mu_P(f)>0$. 
In general, the more singular $H$ is at $P$, the larger the multiplicity and the Milnor number are.
In order to get a feeling for the behavior of these invariants, note that if
$f=x_1^{a_1}+\ldots+x_n^{a_n}$, we have
$${\rm ord}_0(f)=\min_{1\leq i\leq n}a_i,\,\,\mu_0(f)=\prod_{i=1}^n(a_i-1).$$

The Milnor number and other related information
(such as the cohomology of the Milnor fiber, the monodromy action on this cohomology etc) play 
a fundamental role in the topological approach to singularities. However, 
this aspect will not feature much in these notes. The multiplicity, on the other hand,
is a very rough invariant. Nevertheless, it can be very useful: maybe its most spectacular application is in resolution of singularities (see \cite{Kollar_book}), where it motivates and guides the resolution process.

The log canonical threshold $\lct_P(f)$ of $f$ at $P$ is an invariant that, 
as we will explain in \S 1, can be thought of
as a refinement of the reciprocal of the multiplicity. In order to compare its behavior with that of the multiplicity and of the Milnor number, we note that if $f=x_1^{a_1}+\ldots+x_n^{a_n}$,
then $\lct_0(f)=\min\left\{1,\sum_{i=1}^n\frac{1}{a_i}\right\}$.

Several features of the log canonical threshold can be seen on this example: in general,
it is a rational number, it is bounded above by $1$ (in the case of hypersurfaces), and it has
roughly the same size as $1/\ord_P(f)$ (see \S 1 for the precise statement). 
If $H$ is nonsingular at $P$, then $\lct_P(f)=1$. However, we may have
$\lct_P(f)=1$ even when $P\in H$ is a singular point: consider, for example,
$f=x^3+y^3+z^3\in \CC[x,y,z]$. 

The log canonical threshold first appeared implicitly in the paper of Atiyah \cite{Atiyah},
in connection with complex powers. In this paper Atiyah proved the following conjecture of Gelfand. Given $f$ as above, one can easily see that for every $s\in\CC$ with
${\rm Re}(s)>0$ one has a distribution on $\CC^n$ that takes a 
$\CC$-valued smooth function with compact support $\phi$ to $\int_{\CC^n}|f(z)|^{2s}\phi(z)dzd\overline{z}$. I.~M.~Gelfand conjectured
that this can be extended to $\CC$ as a meromorphic map with values in distributions, and Atiyah proved\footnote{At the same time, an independent proof of the same result, based on the same method, was given in \cite{BG}.}  that this is the case using resolution of 
singularities\footnote{In fact, Atiyah's paper and Gelfand's conjecture were in the context of polynomials with real coefficients. We have stated this in the complex case, since it then relates to what we will discuss in \S 1. For a treatment of both the real and the complex case, see 
\cite{Igusa}.}. His proof also shows, with current terminology, that the largest pole is bounded above by $-\lct(f)$, where $\lct(f)=\min_{P\in H}\lct_P(f)$. 

The first properties of the log canonical threshold (known at the time as the
\emph{complex singularity exponent}) have been proved by Varchenko
in connection with his work on
 asymptotic expansions of integrals (similar to the integral we have seen above), and mixed Hodge structures on the vanishing cohomology, see \cite{Varchenko1}, \cite{Varchenko2},
 and \cite{Varchenko3}.  In this context, the log canonical threshold appears
 as one of the numbers in the spectrum of the singularity, a set of invariants due to 
 Steenbrink \cite{Steenbrink}.

It was Shokurov who introduced the log canonical threshold in the context of birational geometry
in \cite{Sho}. In this setting, one thinks of $\lct_P(f)$ as an invariant of the pair 
$(\CC^n,H)$, giving the largest $\lambda>0$ such that the pair
$(\CC^n,\lambda H)$ is \emph{log canonical} in some neighborhood of $P$ (which explains the name). 
We mention that the notion of
log canonical pairs is of central importance in the Minimal Model Program, since it gives 
the largest class of varieties for which one can hope to apply the program. In fact, 
in the context of birational geometry one does not require that the ambient variety is nonsingular,
but only that it has mild singularities, and it is in this more general setting that one can define the log canonical threshold. 
Shokurov made a surprising conjecture, which in the setting of ambient nonsingular varieties
asserts that the set of all log canonical thresholds $\lct_P(f)$, for $f\in\CC[x_1,\ldots,x_n]$
with $n$ fixed, satisfies ACC, that is, it contains no strictly increasing infinite sequences. 
The expectation was that a positive answer to this conjecture (in the general setting of possibly singular varieties) would be related to the
so-called Termination of Flips conjecture in the Minimal Model Program, and Birkar
showed such a relation in \cite{Birkar}. For more on this topic, see \S 3 below.

Meanwhile, it turned out that the log canonical threshold came up in many other
contexts having to do with singularities. The following is an incomplete list of such occurrences,
but which can hopefully give the reader a feeling for the ubiquity of this invariant. 

\noindent $\bullet$ In the case of a polynomial $f\in\ZZ[x_1,\ldots,x_n]$ and of a prime $p$, the
log canonical threshold of $f$ is related to the rate of growth of the number of solutions of $f$ in $\ZZ/p^m$. This is related 
to a $p$-adic analogue of the complex powers discussed above, see \cite{Igusa}.

\noindent $\bullet$ Yet another integration theory (motivic integration) allows one to relate
the log canonical thresholds to the rate of growth of the dimensions of the jet schemes of $X$,
see \cite{Mustata}.

\noindent $\bullet$ The Bernstein polynomial of $f$ is an invariant of the singularities of $f$
that comes out of the theory of $D$-modules. The negative of the log canonical threshold is
the largest root of the Bernstein polynomial, see \cite{Kollar}.

\noindent $\bullet$ Tian's $\alpha$-invariant is an asymptotic version of the log canonical threshold that provides a criterion for the existence of K\"{a}hler-Einstein metrics
(see, for example \cite{Tian}, \cite{DK} and \cite{Cheltsov}). 

\noindent $\bullet$ The log canonical threshold appears implicitly or explicitly in many applications of vanishing theorems, due to its relation to multiplier ideals 
(see \cite[Chapter 9]{Lazarsfeld}). An important example is the work of Angehrn and
Siu \cite{AS} on the global generation of adjoint line bundles. 

\noindent $\bullet$ Lower bounds for the log canonical threshold also come up in proving 
a strong version on non-rationality for certain Fano varieties of index one (for example,
for hypersurfaces of degree $n$ in $\PP^n$). This is a point of view due to Corti
\cite{Corti} on the classical approach to the non-rationality of a quartic threefold of
Iskovskikh and Manin \cite{IM}. See for example \cite{dFEM3} for an application of this point of view.

The present notes are based on a mini-course I gave at the IMPANGA Summer school,
in July 2010. The goal of the lectures was to introduce the log canonical threshold, and present some open problems and recent results related to it. I have tried to 
preserve, as much as possible, the informal character of the lectures, so very few complete
proofs are included. 

In the first section we discuss the definition and some basic properties of the log canonical threshold, as well as some examples. The second section is devoted to an analogous invariant
that comes up in commutative algebra in positive characteristic, the $F$-pure threshold. While
defined in an entirely different way, using the Frobenius morphism, it turns out that this invariant
is related in a subtle way to the log canonical threshold via reduction mod $p$. In Section 3
we discuss a recent joint result with T.~de Fernex and L.~Ein \cite{dFEM}, proving Shokurov's ACC conjecture in the case of ambient smooth varieties. We do not present the details of the proof, but rather describe following \cite{dFM} a key ingredient of the proof, the construction of certain ``limit power series" associated to a sequence of polynomials. 
The last section discusses following \cite{JM} an asymptotic version of the log canonical threshold in the context of graded sequences of ideals, and a basic open question concerning this asymptotic invariant. The content of the first three sections follows roughly the three 
Impanga lectures, while the topic in the fourth section is a subsequent addition, that 
 did not make it into the lectures because of time constraints.

\subsection*{Acknowledgment} I am indebted to the organizers of the IMPANGA 
Summer school for the invitation to give this series of lectures and 
for the encouragement to publish the lecture notes. Special thanks are due to
Tomasz Szemberg for the detailed notes he took during the lectures. During the preparation
of this paper I was partially supported by NSF grant DMS-0758454 and
  by a Packard Fellowship.

\section{Definition and basic properties}

In this section, we will work in the following setting. 
Let $X$ be a nonsingular, irreducible, complex algebraic variety and $\fra\subseteq\cO_X$
a nonzero (coherent)  ideal sheaf (often assumed to be principal). Since we are only interested
in local aspects, we may and will assume that $X=\Spec R$.
 Let $P\in V(\fra)$ be a fixed closed point and $\frm_P$ the corresponding ideal. 
 We refer to a regular system of parameters of
 $\cO_{X,P}$ as 
   \emph{local coordinates} at $P$.
 
 By a \emph{divisor} over $X$ 
we understand a prime divisor $E$ on some \emph{model} $Y$ over $X$, that is, a nonsingular 
variety $Y$ having a projective,
birational morphism $Y\to X$. Every such divisor determines a valuation of the function field
$\CC(Y)=\CC(X)$ that is denoted by $\ord_E$. Explicitly, if $f\in R$ defines the divisor
$D$ on $X$, then ${\rm ord}_E(f)$ is the coefficient of $E$ in $\pi^*(D)$. We also put
$\ord_E(\fra)=\min\{\ord_E(f)\mid f\in\fra\}$. The image of $E$ on $X$ is the
\emph{center} $c_X(E)$ of $E$ on $X$. We identify two divisors over $X$ if they
correspond to the same valuation. 

The \emph{multiplicity} (or \emph{order}) of $\fra$ at $P$ is the largest $r\in\ZZ_{\geq 0}$
such that $\fra\subseteq\frm_P^r$. Of course, we have $\ord_P(\fra)=\min_{f\in\fra}
\ord_P(f)$. It is an easy exercise, using the Taylor expansion, to show that 
if $x_1,\ldots,x_n$ are local coordinates at $P$, then 
$\ord_P(f)$ is the smallest $|\alpha|$ such that $\frac{\partial^{\alpha} f}{\partial x^{\alpha}}(P)$ is nonzero.

 We can rephrase the definition of the order, as follows. If
${\rm Bl}_P(X)\to X$ is the blow-up of $X$ at $P$, and $F$ is the exceptional divisor,
then $\ord_P(\fra)=\ord_F(\fra)$. When defining the log canonical threshold we consider instead
all possible divisors over $X$, not just $F$. On the other hand, we need to normalize somehow
the values $\ord_E(\fra)$, as otherwise these are unbounded. This is done in terms of
log discrepancies.

Consider a projective birational morphism $\pi\colon Y\to X$
   of smooth, irreducible, $n$-dimensional varieties. We have
   the induced sheaf morphism
   $$\pi^*\Omega_X\to\Omega_Y$$
   which induces in turn the nonzero morphism
   $$\pi^*\Omega_X^n\to\Omega_Y^n=\pi^*\Omega_X^n\otimes \cO_Y(K_{Y/X}),$$
   for some effective divisor $K_{Y/X}$, the \emph{relative canonical divisor}, also known as the
   \emph{discrepancy} of $\pi$. 
   Let us show that ${\rm Supp}(K_{Y/X})$ is the inverse image of a closed subset $Z$ of $X$
   of codimension $\geq 2$, such that $\pi$ is an isomorphism over $X\smallsetminus Z$
  (hence the support of $K_{Y/X}$ is the exceptional locus of $\pi$). 
  Indeed, it follows from definition that ${Y\smallsetminus \rm Supp}(K_{Y/X})$ is the set of those $y\in Y$
  such that $\pi$ is \'{e}tale at $y$ (in which case, $y$ is clearly isolated in
  $\pi^{-1}(\pi(y))$). On the other hand, 
  since $\pi$ is birational and $X$ is normal, we have $\pi_*(\cO_Y)=\cO_X$, and
  Zariski's Main Theorem implies that all fibers of $\pi$ are connected. In particular,
  if $y\not\in {\rm Supp}(K_{Y/X})$, then $\pi^{-1}(\pi(y))=\{y\}$. This implies that 
  ${\rm Supp}(K_{Y/X})$ is the inverse image of a subset $Z$ (which is closed in $X$ since 
  $\pi$ is proper). Using the fact that $\pi$ is a homeomorphism over $X\smallsetminus Z$ and
  $\pi_*(\cO_Y)=\cO_X$, we deduce that $\pi$ is an isomorphism over $X\smallsetminus Z$. 
  Since $X$ is normal and $Y$ is proper over $X$, it follows that $\pi^{-1}$ is defined in codimension one, which easily implies ${\rm codim}(Z,X)\geq 2$. 
   
    Given a divisor $E$ over $X$ lying on
   the model $Y$ over $X$, the
   \emph{log discrepancy} of $E$ is
   ${\rm Logdisc}(E):=1+\ord_E(K_{Y/X})$. It is easy to see that the definition is independent
   on the particular model $Y$ we have chosen.
   
   The \emph{Arnold multiplicity} of the nonzero ideal $\fra$ at $P\in V(\fra)$ is defined as
   \begin{equation}\label{Arnold_multiplicity}
{\rm Arn}_P(\fra)=\sup_E\frac{\ord_E(\fra)}{{\rm Logdisc}(E)},
\end{equation}
where the supremum is over the divisors $E$ over $X$ such that $P\in c_X(E)$.
Note that we may consider the Arnold multiplicity as a more subtle version of the usual multiplicity.
The log canonical threshold is the reciprocal of the Arnold multiplicity:
$\lct_P(\fra)=1/{\rm Arn}_P(\fra)$. 

It is clear that $\ord_E(\fra)>0$ if and only if $c_X(E)$ is contained in $V(\fra)$. By taking any divisor $E$ with center $P$, we see that
${\rm Arn}_P(\fra)$ is positive, hence $\lct_P(\fra)$ is finite. We make the convention that
$\lct_P(\fra)=\infty$ if $P\not\in V(\fra)$.
We will see in Property~\ref{lower_bound} below that since
$\fra$ is assumed nonzero, we have $\lct_P(\fra)>0$.

Intuitively, the worse a singularity is, the higher the multiplicities $\ord_E(\fra)$ are,
and therefore the higher ${\rm Arn}_P(\fra)$ is, and consequently the smaller 
$\lct_P(\fra)$ is. We will  illustrate this by some examples in 
\S 1.2 below.

\subsection{Analytic interpretation and computation via resolution of singularities}

What makes the above invariant computable is the fact that it can be described in terms of a log resolution of singularities. Recall that a projective, birational morphism 
$\pi\colon W\to X$, with $W$ nonsingular, is a log resolution of $\fra$ if the inverse image
$\fra\cdot\cO_W$ is the ideal of a Cartier divisor $D$ such that $D+K_{Y/X}$ is a divisor
with simple normal crossings. This means that at every point $Q\in W$ there are
local coordinates $y_1,\ldots,y_n$ such that $D+K_{Y/X}$ is defined
by $(y_1^{\alpha_1}\cdot\ldots\cdot y_n^{\alpha_n})$, for some $\alpha_1,\ldots,\alpha_n
\in\ZZ_{\geq 0}$. It is a consequence of Hironaka's theorem on resolution of singularities
that log resolutions exist in characteristic zero. Furthermore, since $X$ is nonsingular,
whenever it is convenient we may assume that $\pi$ is an isomorphism over the complement of 
$V(\fra)$.

The following theorem, that can be viewed as a finiteness result, is fundamental for working with log canonical thresholds. 

\begin{theorem}\label{compute_by_resolution}
Let $f\colon W\to X$ be a log resolution of $\fra$, and consider a divisor with simple normal crossings
$\sum_{i=1}^NE_i$ on $W$ such that if $\fra\cdot\cO_W=\cO_W(-D)$, then
we may write
$$D=\sum_{i=1}^Na_iD_i\,\,\text{and}\,\,K_{W/X}=\sum_{i=1}^Nk_iE_i.$$
In this case, we have 
\begin{equation}\label{eq_compute_by_resolution}
\lct_P(\fra)=\min_{i: P\in\pi(E_i)}\frac{k_i+1}{a_i}.
\end{equation}
\end{theorem}

One can give a direct algebraic proof of the above theorem: since every divisor over $X$
appears on some log resolution of $\fra$, the assertion in the theorem is equivalent with the 
fact that the expresion in (\ref{eq_compute_by_resolution}) does not depend on the choice of 
resolution. For the proof of this statement, see \cite[Theorem~9.2.18]{Lazarsfeld}.

We prefer to give a different argument, involving an analytic description for the log canonical threshold. The advantage of this result is that it provides some more intuition for the log canonical threshold, making also the connection with the way it first appeared in the context of complex powers mentioned in the Introduction.

\begin{theorem}\label{analytic_description}
If $\fra=(f_1,\ldots,f_r)$ is a nonzero ideal on the smooth, irreducible, complex affine algebraic variety
$X=\Spec R$, for every point $P\in X$ we have
$$\lct_P(\fra)=\sup\left\{s>0\mid \frac{1}{\left(\sum_{i=1}^r|f_i|^2\right)^s}\,\,
\text{is integrable around}\,P\right\}.$$
\end{theorem}

\begin{proof}[Sketch of proof of Theorems~\ref{analytic_description} and
\ref{compute_by_resolution}]
The assertions in both theorems follow if we show that given a log resolution 
$\pi\colon W\to X$ of $\fra$
as in Theorem~\ref{compute_by_resolution}, we have
\begin{equation}\label{eq_analytic_description}
\frac{1}{\left(\sum_{i=1}^r|f_i|^2\right)^s}\,\,
\text{is integrable around}\,P\,\,\text{iff}\,\,s<\frac{k_i+1}{a_i}\,\,\text{for all}\,\,i\,\,\text{with}\,P\in\pi(E_i).
\end{equation}

Let us choose local coordinates $z_1,\ldots,z_n$ at $P$. Of course, for integrability questions we consider the corresponding structure of complex manifold on $X$. 
In particular, we say that a positive real  function $h$ is integrable 
around $P$ if for some 
 open subset (in the classical topology) $U\subseteq X$ containing $P$, we have
$\int_Uh\,dzd\overline{z}<\infty$ (it is easy to see that this is independent of the choice of coordinates). 

The key point is that the change of variable formula implies
\begin{equation}\label{eq2_analytic_description}
\int_U\frac{1}{\left(\sum_{i=1}^r|f_i|^2\right)^s} dzd\overline{z}=
\int_{\pi^{-1}(U)}\frac{1}{\left(\sum_{i=1}^r|f_i\circ\pi|^2\right)^s} \pi^*(dz)\pi^*(d\overline{z}).
\end{equation}
This is due to the fact that there is an open subset $V\subseteq X$ such that 
$\pi$ is an isomorphism over $V$, and $U\smallsetminus V\subset U$ and
$\pi^{-1}(U)\smallsetminus\pi^{-1}(V)\subset\pi^{-1}(U)$ are proper closed analytic
subsets, thus have measure zero.

It is easy to see that given a finite open cover $\pi^{-1}(U)=\bigcup_jV_j$, the finiteness of the right-hand
side of (\ref{eq2_analytic_description}) is equivalent to the finiteness of the integrals of the same function on each of the $V_j$. Suppose that on $V_j$ we have coordinates
$y_1,\ldots,y_n$ with the following properties: $K_{V_j/X}$ is defined by
$(y_1^{k_1}\cdot\ldots\cdot y_n^{k_n})$ and $\fra\cdot\cO_{V_j}$ is generated
by $y_1^{a_1}\cdot\ldots\cdot y_n^{a_n}$. Since $\pi$ is a log resolution, we see
that we may choose a cover as above, such that on each $V_j$ we can find such a system
of coordinates.

We can thus write on $V_j$
$$f_i\circ\pi=u_iy_1^{a_1}\cdots y_n^{a_n},$$
for some regular functions $u_1,\ldots,u_r$ on $V_j$, with no common zero. 
We also see that 
$$\pi^*(dz)\pi^*(d\overline{z})=wy_1^{2k_1}\cdot\ldots\cdot y_n^{2k_n} dyd\overline{y},$$
for some invertible regular function $w$ on $V_j$. We conclude that
\begin{equation}\label{eq3_analytic_description}
\int_{V_j}\frac{1}{\left(\sum_{i=1}^r|f_i\circ\pi|^2\right)^s} \pi^*(dz)\pi^*(d\overline{z})=
\int_{V_j}\frac{w}{\left(\sum_{i=1}^r|u_i|^2\right)^s}\prod_{i=1}^n|y_i|^{2k_i-2sa_i}dy d\overline{y}.
\end{equation}
Since $\pi$ is proper, $\pi^{-1}(K)$ is compact for every compact subset $K$ of $X$.
One can show that by a suitable choice of $U$ and of the $V_j$, we may assume that each
$\overline{V_j}$ is compact, and both $w$ and $\sum_{i=1}^r|u_i|^2$ extend to
invertible functions on $\overline{V_j}$. In particular, the right-hand side of (\ref{eq3_analytic_description}) is finite of and only if 
\begin{equation}\label{eq4_analytic_description}
\int_{V_j}\prod_{i=1}^n|y_i|^{2k_i-2sa_i}dy d\overline{y}<\infty.
\end{equation}

On the other hand, it is well-known that $\int_{U'}|z|^{\alpha}dz d\overline{z}<\infty$
for some neighborhood of the origin $U'\subseteq\CC$
if and only if $\alpha>-2$. This implies via Fubini's theorem that
(\ref{eq4_analytic_description}) holds if and only if 
$2k_i-2sa_i>-2$ for all $i$. Since we are allowed to replace $U$ by a small neighborhood of $P$,
the $k_i$ and $a_i$ that we see in the above conditions when we vary the $V_j$ correspond precisely to those divisors $E_i$ whose image contains $P$. We thus get the formula
(\ref{eq_analytic_description}).
\end{proof}

\begin{remark}
One consequence of Theorem~\ref{compute_by_resolution} is that $\lct_P(\fra)$ is a rational number. Note that the definition of the log canonical threshold makes sense also in positive characteristic, but the rationality of the invariant in not known in that context. 
\end{remark}

There is also a global version of the log canonical threshold and of Arnold multiplicity:
$$\lct(\fra)=\min_{P\in X}\lct_P(\fra)\,\,\text{and}\,\,{\rm Arn}(\fra)=\max_{P\in X}{\rm Arn}_P(\fra).$$
With the notation in Theorem~\ref{compute_by_resolution}, we see that
$$\lct(\fra)=\min_{i}\frac{k_i+1}{a_i}.$$
By definition, $\lct(\fra)$ is infinite if and only if $\fra=\cO_X$.
Note also  that we have $\lct_P(\fra)=\max_{U\ni P}\lct(U,\fra\vert_U)$, where
$U$ varies over the open neighborhoods of $P$.

\subsection{Examples of log canonical threshold computations}

In this subsection we collect some easy examples of log canonical thresholds.
For details and further examples, we refer to \cite[Chapter 9]{Lazarsfeld}. 
\begin{example}
Suppose that $\fra=(f)$ is the ideal defining a nonsingular hypersurface. In this case, the identity map on $X$ gives a log resolution of $\fra$, hence by
Theorem~\ref{compute_by_resolution} we have $\lct_P(f)=1$ for every $P\in V(f)$. 
\end{example}

\begin{example}\label{smooth_subscheme}
More generally, suppose that $\fra$ is the ideal defining a nonsingular subscheme $Z$ of pure codimension $r$. The blow-up $W\to X$ of $X$ along $Z$
gives a log resolution of $\fra$, with $K_{W/X}=(r-1)E$, where $E$ is the exceptional divisor
(check this!). It follows from Theorem~\ref{compute_by_resolution} that
$\lct_P(\fra)=r$ for every $P\in Z$. 
In particular, if $\frm_P$ is the ideal defining $P$, we see that $\lct_P(\frm_P)=\dim(X)$.
\end{example}

\begin{example}\label{bounded_by_1}
If $f\in\cO(X)$ is such that the divisor of $f$ is $\sum_{i=1}^ra_iD_i$, then by taking 
$E=D_i$ in the definition of the log canonical threshold, we conclude that if $P\in V(f)$, then
$$\lct_P(f)\leq\min_{i: P\in D_i}\frac{1}{a_i}\leq 1.$$
\end{example}

\begin{example}
Suppose that $f\in\CC[x,y]$ has a node at $P$. In this case the blow-up $W$ of $\AAA^2$ at $P$
gives a log resolution of $f$ in some neighborhood of $P$, and the inverse image 
of $V(f)$ is $D+E$, where $D$ is the proper transform, and $E$ is the exceptional divisor.
Since $K_{W/\AAA^2}=E$, it follows from Theorem~\ref{compute_by_resolution}
that $\lct_P(f)=1$. 
\end{example}

\begin{example}\label{homogeneous}
Let $f\in\CC[x_1,\ldots,x_n]$ be a homogeneous polynomial
of degree $d$, having an isolated singularity at the origin. If $\pi\colon W\to\AAA^n$ is the blow-up of the origin, and $E$ is the exceptional divisor, then $K_{W/\AAA^n}=(n-1)E$ and $f\cdot\cO_W=\cO(-D-dE)$, where $D$ is the proper transform of $V(f)$. Note that we have 
an isomorphism $E\simeq\PP^{n-1}$ such that $D\cap E$ is isomorphic to the projective hypersurface defined by $f$, hence it is nonsingular. Therefore $D+E$ is a divisor with  simple normal crossings, and we see that $\pi$ is a log resolution of $(f)$. It follows from 
Theorem~\ref{compute_by_resolution} that
$\lct(f)=\lct_0(f)=\min\left\{1,\frac{n}{d}\right\}$.
\end{example}

\begin{example}\label{monomial}
Suppose that $\fra\subset\CC[x_1,\ldots,x_n]$ is a proper
nonzero ideal generated by monomials. For $u\in\ZZ_{\geq 0}^n$, we write
$x^u=x_1^{u_1}\cdots x_n^{u_n}$. Given $u=(u_1,\ldots,u_n)$ and $v=(v_1,\ldots,v_n)$ in
$\RR^n$, we put $\langle u,v\rangle=\sum_{i=1}^nu_iv_i$.

The \emph{Newton polyhedron} of $\fra$
is 
$$P(\fra)={\rm convex}\,{\rm hull}\left(\{u\in\ZZ_{\geq 0}^n\mid x^u\in\fra\}\right).$$
Howald showed in \cite{Howald} that
$$\lct(\fra)=\lct_0(\fra)=\max\{\lambda\in\RR_{\geq 0}\mid (1,\ldots,1)\in\lambda\cdot
P(\fra)\}.$$
This follows rather easily using some basic facts about toric varieties (for these facts,
see \cite{Fulton}). Indeed, if we consider the standard toric structure on $\AAA^n$, the fact that
$\fra$ is generated by monomials says precisely that the $(\CC^*)^n$-action on $\AAA^n$ induces 
an action on the closed subscheme defined by $\fra$. By blowing up $\AAA^n$ along $\fra$,
and then taking a toric resolution of singularities, we see that we can find a projective, birational 
morphism of toric varieties $\pi\colon W\to X$ that gives a log resolution of $\fra$
(indeed, in this case both $K_{W/X}$ and the divisor corresponding to $\fra\cdot\cO_W$
are toric, hence have simple normal crossings, since $W$ is nonsingular). Theorem~\ref{compute_by_resolution} implies that in the definition of the log canonical threshold it is enough
to consider torus invariant divisors on toric varieties $Y$ having projective, birational, toric morphisms
to $X$. Every such divisor $E$ corresponds to a primitive nonzero integer vector
$v=(v_1,\ldots,v_n)\in\ZZ_{\geq 0}^n$ such that 
$$\ord_E(\fra)=\min\{\langle u,v\rangle\mid u\in P(\fra)\}\,\,\text{and}\,\,
{\rm Logdisc}(E)=v_1+\ldots+v_n.$$
Therefore $\lct(\fra)$ is equal to the largest $\lambda$ such that
$\sum_{i=1}^nv_i\geq\lambda\cdot \min_{u\in P(\fra)}\langle u,v\rangle$ for every
$v\in\ZZ_{\geq 0}^n$ primitive and nonzero (equivalently, for every $v\in\QQ_{\geq 0}^n$).
It is then easy to see that this is equivalent to $(1,\ldots,1)\in \lambda \cdot P(\fra)$.

For example, suppose that $\fra=(x_1^{a_1},\ldots,x_n^{a_n})$. It follows from definition that
$P(\fra)=\left\{(u_1,\ldots,u_n)\in\RR_{\geq 0}^n\mid\sum_{i=1}^n\frac{u_i}{a_i}\geq 1\right\}$.
 Howald's formula gives in this case $\lct_0(\fra)=\sum_{i=1}^n\frac{1}{a_i}$.
 \end{example}

\begin{example}\label{general_coefficients}
Let $\fra=(f_1,\ldots,f_r)$, and consider $f=\sum_{i=1}^r\lambda_if_i$, where 
$\lambda_1,\ldots,\lambda_r$ are general complex numbers. Consider a log resolution
$\pi\colon W\to X$ of $\fra$ that is an isomorphism over $X\smallsetminus V(\fra)$, and write $\fra\cdot\cO_W=\cO_W(-D)$.
In this case $f\cdot\cO_W=\cO_W(-D-F)$, for some divisor $F$, and it is an easy consequence 
of Bertini's theorem that $F$ is nonsingular and $F+D$ has simple normal crossings. If we write
$$D=\sum_{i=1}^Na_iE_i\,\,\,\text{and}\,\,\,K_{W/X}=\sum_{i=1}^N k_iE_i,$$
then $\ord_{E_i}(F)=0$ if $a_i>0$, 
and we have $a_i\in\{0,1\}$ for all $i$. Since $\pi$ is an isomorphism over the complement of
$V(\fra)$, it follows that $k_i=0$ if $a_i=0$. 
We then conclude from
Theorem~\ref{compute_by_resolution} that $\lct_P(f)=\min\{\lct_P(\fra),1\}$. 
\end{example}

\begin{example}\label{diagonal}
Let $f=x_1^{a_1}+\ldots+x_n^{a_n}$, and consider $\fra=(x_1^{a_1},\ldots, x_n^{a_n})$. 
Given any nonzero $\lambda_1,\ldots,\lambda_n$, there is an isomorphism of $\AAA^n$
(leaving the origin fixed) that takes $f$ to $\sum_{i=1}^n\lambda_ix_i^{a_i}$.
It follows from Examples~\ref{monomial} and \ref{general_coefficients} that
$\lct_0(f)=\min\left\{1,\sum_{i=1}^n\frac{1}{a_i}\right\}$.
\end{example}

\subsection{Basic properties}
We give a brief overview of the main properties of the log canonical threshold. 
For some applications of the log canonical threshold in
birational geometry we refer to the survey \cite{EM}.

\begin{property}\label{monotonicity}
If $\fra\subseteq \frb$ are nonzero ideals on $X$, then $\lct_P(\fra)\leq\lct_P(\frb)$
for every $P\in X$. Indeed, the hypothesis implies that $\ord_E(\fra)\geq\ord_E(\frb)$
for every divisor $E$ over $X$. 
\end{property}

\begin{property}\label{power}
We have $\lct_P(\fra^r)=\frac{\lct_P(\fra)}{r}$ for every $r\geq 1$. Indeed, for every divisor 
$E$ over $X$ we have $\ord_E(\fra^r)=r\cdot\ord_E(\fra)$.
\end{property}

\begin{property}\label{upper_bound}
For every ideal $\fra$ on $X$, we have $\lct_P(\fra)\leq \frac{n}{\ord_P(\fra)}$,
where $n=\dim(X)$ (note that by convention, both sides are infinite if
$P\not\in V(\fra)$). The assertion follows from the fact that if $r=\ord_P(\fra)$
(which we may assume to be positive),
then $\fra\subseteq\frm_P^r$, where $\frm_P$ is the ideal defining $P$. Using 
Example~\ref{smooth_subscheme} and Properties~\ref{monotonicity} and \ref{power}, we conclude
$$\lct_P(\fra)\leq \lct_P(\frm_P^r)=\frac{\lct_P(\frm_P)}{r}=\frac{n}{r}.$$
\end{property}

\begin{property}
If $\overline{\fra}$ is the integral closure of $\fra$, then
$\lct(\fra)=\lct(\overline{\fra})$ (see \cite[\S 11.1]{Lazarsfeld} for definition and
basic properties of integral closure). The key point is that for every divisor $E$ over 
$X$, we have $\ord_E(\fra)=\ord_E(\overline{\fra})$.
\end{property}

\begin{property}\label{convexity}
If $\fra$ and $\frb$ are ideals on $X$, then
\begin{equation}\label{eq_convexity}
\Arn(\fra\cdot\frb)\leq \Arn(\fra)+\Arn(\frb).
\end{equation}
Indeed, for every divisor $E$ over $X$ we have
$$\frac{\ord_E(\fra\cdot\frb)}{{\rm Logdisc}(E)}=\frac{\ord_E(\fra)}{{\rm Logdisc}(E)}+
\frac{\ord_E(\frb)}{{\rm Logdisc}(E)}\leq\Arn(\fra)+\Arn(\frb).$$
By taking the maximum over all $E$, we get (\ref{eq_convexity}).
\end{property}

\begin{property}\label{inversion_of_adjunction}
If $H\subset X$ is a nonsingular hypersurface such that $\fra\cdot\cO_H$ is nonzero,
then $\lct_P(\fra\cdot\cO_H)\leq\lct_P(\fra)$ for every $P\in H$. Note that this is compatible 
with the expectation that the singularities of $\fra$ are at least as good as those of
$\fra\cdot\cO_H$. This is one of the more subtle properties of log canonical thresholds,
that is known as \emph{Inversion of Adjunction}. It can be proved using either vanishing
theorems (see \cite[Theorem~9.5.1]{Lazarsfeld}), or the description of the log canonical threshold in terms of jets schemes (see \cite[Proposition~4.5]{Mustata}). 

More generally, if $Y\hookrightarrow X$ is a nonsingular
closed subvariety such that $\fra\cdot\cO_Y$ is nonzero, then 
$\lct_P(\fra\cdot\cO_Y)\leq\lct_P(\fra)$ for every $P\in Y$. This follows by a repeated application of the codimension one case, by realizing $Y$ in some neighborhood of $P$ as
$H_1\cap\ldots\cap H_r$, where $r={\rm codim}_X(Y)$ (note that in this case each
$H_1\cap\ldots\cap H_i$ is nonsingular at the points in $Y$). 
\end{property}

\begin{property}\label{lower_bound}
For every point $P\in X$, we have $\lct_P(\fra)\geq\frac{1}{\ord_P(\fra)}$.
This is proved by induction on $\dim(X)$ using Property~\ref{inversion_of_adjunction}.
Indeed, if $\dim(X)=1$ and $t$ is a local coordinate at $P$,
then around $P$ we have $\fra=(t^r)$, where $r=\ord_P(\fra)$, while $\lct_P(\fra)=1/r$.
For the induction step, note that if $x_1,\ldots,x_n$ are local coordinates at $P$, and if $H$ is defined by $\lambda_1x_1+\ldots+\lambda_nx_n$, with
$\lambda_1,\ldots,\lambda_n\in\CC$ general, then $H$ is nonsingular at $P$, and $\ord_P(\fra)=\ord_P(\fra\cdot\cO_H)$,
while $\lct_P(\fra)\geq\lct_P(\fra\cdot\cO_H)$. 
\end{property}

\begin{property}\label{product}
If $X$ and $Y$ are nonsingular varieties, and $\fra$ and $\frb$ are nonzero ideals on
$X$ and $Y$, respectively, then
$$\lct_{(P,Q)}(p^{-1}(\fra)+q^{-1}(\frb))=\lct_P(\fra)+\lct_Q(\frb)$$
for every $P\in X$ and $Q\in Y$, where $p\colon X\times Y\to X$ and $q\colon X\times Y\to Y$
are the canonical projections. This can be proved either as a consequence of the
Summation Formula for multiplier ideals (see \cite[Theorem~9.5.26]{Lazarsfeld})
or using the description of the log canonical threshold in terms of jet schemes 
(see \cite[Proposition~4.4]{Mustata}).
\end{property}

\begin{property}\label{sum_of_ideals}
If $\fra$ and $\frb$ are ideals on $X$, then
$$\lct_P(\fra+\frb)\leq\lct_P(\fra)+\lct_P(\frb)$$
for every $P\in X$. Indeed, we may apply Property~\ref{inversion_of_adjunction}
(in its general form) to the subvariety $X\hookrightarrow X\times X$, embedded diagonally.
Indeed, using also Property~\ref{product} we get
$$\lct_P(\fra+\frb)\leq \lct_{(P,P)}(p^{-1}(\fra)+q^{-1}(\frb))=\lct_P(\fra)+\lct_P(\frb).$$
\end{property}

\begin{property}\label{truncation}
If $\frm_P$ is the ideal defining a point $P\in X$, and $\fra+\frm_P^N=\frb+\frm_P^N$, then
$$|\lct_P(\fra)-\lct_P(\frb)|\leq \frac{n}{N},$$
where $n=\dim(X)$. Indeed, using Properties~\ref{monotonicity}, \ref{sum_of_ideals}
and \ref{power}, we obtain
$$\lct_P(\frb)\leq\lct_P(\frb+\frm_P^N)=\lct_P(\fra+\frm_P^N)\leq\lct_P(\fra)+\lct_P(\frm^N_P)=\lct_P(\fra)+\frac{n}{N}.$$
By symmetry, we also get $\lct_P(\fra)\leq\lct_P(\frb)+\frac{n}{N}$. 

In particular, if $f_{\leq N}\in\CC[x_1,\ldots,x_n]$ is the truncation of  $f$ up
to degree $\leq N$, then $|\lct_0(f)-\lct_0(f_{\leq N})|\leq\frac{n}{N+1}$. 
\end{property}

\begin{property}\label{multiplicity}
Suppose that $\fra$ is an ideal supported at a point on the smooth $n$-dimensional complex variety $X$.
In this case we have the following inequality relating the Hilbert-Samuel multiplicity $e(\fra)$ of
$\fra$ to the log canonical threshold:
\begin{equation}\label{eq_multiplicity}
e(\fra)\geq\frac{n^n}{\lct(\fra)^n}.
\end{equation}
This is proved in \cite{dFEM4} by first proving a similar inequality for length:
\begin{equation}
\ell(\cO_X/\fra)\geq\frac{n^n}{n!\lct(\fra)^n}.
\end{equation}
This in turn follows by considering a Gr\"{o}bner deformation of $\fra$ to a monomial ideal,
for which the inequality follows from the combinatorial description of both
$\ell(\cO_X/\fra)$ and $\lct(\fra)$.

Suppose, for example, that $\fra=(x_1^{a_1},\ldots,x_n^{a_n})\subseteq\CC[x_1,\ldots,x_n]$.
It is easy to see, using the definition, that $e(\fra)=a_1\cdots a_n$, while
Example~\ref{monomial} implies that $\lct(\fra)=\sum_{i=1}^n\frac{1}{a_i}$.
Therefore the inequality (\ref{eq_multiplicity}) becomes
$$\frac{\sum_{i=1}^n\frac{1}{a_i}}{n}\geq \frac{1}{(a_1\cdots a_n)^{1/n}},$$
that is, the inequality between the arithmetic mean and the geometric mean.
\end{property}

\begin{property}\label{families1}
Suppose that $U$ is an affine variety, and $\fra\subseteq
\cO(U)[x_1,\ldots,x_n]$ is an ideal contained in $(x_1,\ldots,x_n)$. 
For every $t\in U$,
we consider $\fra_t\subset\CC(t)[x_1,\ldots,x_n]\simeq\CC[x_1,\ldots,x_n]$.
There is a disjoint decomposition of $U$ into finitely many locally closed subsets
$Z_1,\ldots,Z_d$, and $\alpha_1,\ldots,\alpha_d$ such that for every $t\in Z_i$ we have 
$\lct_0(\fra_t)=\alpha_i$. Indeed, if $\pi\colon {\mathcal Y}\to U\times \AAA^n$
is a log resolution of $\fra$,
then it follows from Generic Smoothness that
there is an open subset $U'\subseteq U$ such that for every 
$t\in U'$, if ${\mathcal Y}_t$ is the fiber of 
${\mathcal Y}$ over $t$, the induced morphism $\pi_t\colon {\mathcal Y}_t\to \AAA^n$
gives a log resolution of $\fra_t$ in a neighborhood of $0$. In particular, $\lct_0(\fra_t)$ is independent of $t\in U'$. After repeating this argument for an affine cover of $U\smallsetminus U'$, we obtain the desired cover.
\end{property}

\begin{property}\label{families2}
A deeper property is the semicontinuity of the log canonical threshold. This 
says that in the context described in Property~\ref{families1},
for every   $t\in U$, there is an open neighborhood $W$ of $t$ such that
$\lct_0(\fra_{t'})\geq\lct_0(\fra_t)$ for every  $t'\in W$. This was first proved by \cite{Varchenko1}. For other proofs,
see \cite[Corollary~9.5.39]{Lazarsfeld}, \cite[Theorem~3.1]{DK} and 
\cite[Theorem~4.9]{Mustata}.
\end{property}

\begin{property}\label{mu_constant}
Suppose now that we are in the context of Property~\ref{families1}, but
$\fra=(f)$ is a principal ideal, such that for every  $t\in U$, the
polynomial $f_t$ has an isolated singularity at $0$. If $U$ is connected and the Milnor
number $\mu(f_t)$ is constant for $t\in U$, then also the log canonical threshold
$\lct_0(f_t)$ is constant. The only proof for this fact is due to Varchenko \cite{Varchenko3}.
It relies on the fact that the log canonical threshold is one of the numbers in the spectrum of the singularity. One shows that all the spectral numbers satisfy a semicontinuity property analogous
to Property~\ref{families2}. Since the sum of the spectral numbers is the Milnor number,
and this is constant, these spectral numbers, and in particular the log canonical threshold, are constant.
\end{property}

\subsection{The connection with multiplier ideals}
A natural setting for studying the log canonical threshold is provided by multiplier ideals. 
In what follows we only give the definition and explain the connection with the log canonical threshold. For a thorough introduction to the theory of multiplier ideals, we refer to 
\cite[Chapter 9]{Lazarsfeld}.

As above, we consider a nonsingular, irreducible, affine complex algebraic variety
$X=\Spec R$. Let $\fra=(f_1,\ldots,f_r)$ be a nonzero ideal on $X$. For every $\lambda\in\RR_{\geq 0}$, the \emph{multiplier ideal}
$\cJ(\fra^{\lambda})$ consists of all $h\in R$ such that for every divisor $E$
over $X$, we have
\begin{equation}\label{condition_multiplier}
\ord_E(h)>\lambda\cdot\ord_E(\fra)-{\rm Logdisc}(E).
\end{equation}
In fact, in analogy with Theorem~\ref{compute_by_resolution}, one can show that
it is enough to consider only those divisors $E$ lying on a log resolution of $\fra$. 
One also has the following analytic description of multiplier ideals:
$$h\in\cJ(\fra^{\lambda})\,\,\text{iff}\,\,\frac{|h|^2}{\left(\sum_{i=1}^r|f_i|^2\right)^{\lambda}}\,\,
\text{is locally integrable}.$$
Again, one can prove both these statements at the same time, arguing as in the proof we have sketched for Theorems~\ref{compute_by_resolution} and \ref{analytic_description}. 
Since we only need to check conditions given by finitely many divisors, it is easy to show
that the definition commutes with localization at a nonzero element in $R$, hence we get in this way coherent ideals on $X$.

We have $\fra^m\subseteq\cJ(\fra^m)$ for every $m\in\ZZ_{\geq 0}$.
It is clear from definition that if $\lambda<\mu$, then $\cJ(\fra^{\mu})\subseteq
\cJ(\fra^{\lambda})$. Furthermore, since it is enough to check the condition
(\ref{condition_multiplier}) for only finitely many divisors $E$, it follows that given any
$\lambda$, there is $\epsilon>0$ such that
$\cJ(\fra^{\lambda})=\cJ(\fra^t)$ for every $t$ with $\lambda\leq t\leq\lambda+\epsilon$. 

A positive $\lambda$ is a \emph{jumping number}  of $\fra$ if 
$\cJ(\fra^{\lambda})\neq\cJ(\fra^{\lambda'})$ for every $\lambda'<\lambda$. 
Note that this is the case if and only if there is  $h\in \cJ(\fra^{\lambda})$ and
a divisor $E$ over $X$ such that
$$\ord_E(h)+{\rm Logdisc}(E)=\lambda\cdot\ord_E(f).$$
In particular, it follows that all jumping numbers are rational. Furthermore, since we may consider only the divisors lying on a log resolution of $\fra$, the denominators of the jumping numbers
are bounded, hence the set of jumping numbers is a discrete set of rational numbers.

By definition, $\cJ(\fra^{\lambda})=\cO_X$ if and only if 
$\lambda<\frac{{\rm Logdisc}(E)}{\ord_E(\fra)}$ for all divisors $E$, that is,
$\lambda<\lct(\fra)$. Therefore the smallest jumping number is the log canonical threshold
$\lct(\fra)$. The properties of the log canonical threshold discussed in the previous subsection
have strengthening at the level of multiplier ideals. We refer to \cite[Chapter 9]{Lazarsfeld} for this 
circle of ideas.

If $\fra=(f)$ is a principal ideal, then it is easy to see that
for every $\lambda\geq 1$ we have $\cJ(f^{\lambda})\subseteq (f)$
(consider the condition in the definition when $E$ runs over the irreducible components
of $V(f)$). Furthermore, it follows from definition that $fh\in\cJ(f^{\lambda})$
if and only if $h\in\cJ(f^{\lambda-1})$, hence $\cJ(f^{\lambda})=f\cdot\cJ(f^{\lambda-1})$
for every $\lambda\geq 1$. In particular, this implies that
$\lambda\geq 1$ is a jumping number if and only if $\lambda-1$ is a jumping number.

A deeper fact, known as Skoda's theorem, says that for every ideal $\fra$, we have
\begin{equation}\label{Skoda_equality}
\cJ(\fra^{\lambda})=\fra\cdot\cJ(\fra^{\lambda-1})
\end{equation}
for every $\lambda\geq n=\dim(X)$. The proof of this fact uses vanishing theorems, see
\cite[Chapter~9.6.C]{Lazarsfeld}. The name is due to the fact that (\ref{Skoda_equality}) easily implies the theorem
of Brian\c{c}on-Skoda \cite{BS}. Indeed, since every multiplier ideal is integrally closed
(this is an immediate consequence of the definition (\ref{condition_multiplier})), 
the integral closure of $\fra^n$ is contained in $\fra$:
$$\overline{\fra^n}\subseteq\overline{\cJ(\fra^n)}=\cJ(\fra^n)\subseteq\fra.$$
It is interesting to note that while the proof in \cite{BS} relies on some analytic results
obtained by Skoda via $L^2$ methods, and the proof in \cite{Lazarsfeld} makes use of vanishing
theorems, another proof of the Brian\c{c}on-Skoda theorem was obtained by
Hochster and Huneke in \cite{HH} via characteristic $p$ methods. We now turn to a different instance of such a connection between these three circles of ideas.

\section{Connections with positive characteristic invariants}

In this section we describe an invariant defined in positive characteristic
using the Frobenius morphism, the $F$-pure threshold.
As we will see, this invariant satisfies properties similar to those of the log canonical threshold,
and it is related with this one in a subtle way via reduction mod $p$.

The $F$-pure threshold has been introduced by Takagi and Watanabe \cite{TW}
when the ambient variety is fairly general. In what follows we will focus on the case of
ambient nonsingular varieties, in which case we can use a more direct asymptotic definition,
following \cite{MTW}.

Let $k$ be a perfect\footnote{A more natural condition in this context is the weaker condition  that $k$ is $F$-finite, that is, 
$[k\colon k^p]<\infty$.} field of positive characteristic $p$. We consider a
regular, finitely generated algebra $R$ over $k$,
and let
$X=\Spec R$. 
We denote by $F\colon R\to R$ the Frobenius morphism on $R$ that takes $u$ to $u^p$.
Note that since $k$ is perfect (or, more generally, when $k$ is $F$-finite), the morphism $F$ is finite. Since $R$ is nonsingular, $F$ is also flat. Indeed, it is enough to show that the induced morphism on 
the completion $\widehat{\cO_{X,Q}}$ is flat for every $Q\in X$; since this local ring is isomorphic to
$k(Q)\llbracket x_1,\ldots,x_r\rrbracket$, where $k(Q)$ is the residue field of $Q$, the Frobenius morphism is easily seen to be flat. Therefore $R$ is projective as an $R$-module via $F$.

Let $\fra\subseteq R$ be a nonzero ideal, and $P\in V(\fra)$ a 
closed\footnote{The restriction to closed points does not play any role. We make it in order for some statements to parallel those in \S 1.} point defined by the maximal ideal 
$\frm_P\subset R$. 
Before defining the $F$-pure threshold, let us consider the following description of
$\ord_P(\fra)$ (which also works in characteristic zero). For every integer $r\geq 1$,
let
$$\alpha(r):=\,\mbox{ largest } i\,\mbox{ such that } \fra^i\not\subseteq\frm_P^r.$$
The condition $\fra^i\not\subseteq\frm_P^r$ is satisfied precisely when
$i\cdot\ord_P(\fra)<r$, hence
$$\alpha(r)=\roundup{\frac{r}{\ord_P(\fra)}}-1.$$
Therefore we have $\lim_{r\to\infty}\frac{\alpha(r)}{r}=\frac{1}{\ord_P(\fra)}$.

We get the $F$-pure threshold by a similar procedure, replacing the usual powers of
$\frm_P$ by Frobenius powers. Recall that for every ideal $I$ and every $e\geq 1$
$$I^{\left[p^e\right]}\,=\,\left(h^{p^e}\; | h\in I\right).$$
   If $I$ is generated by $h_1,\ldots,h_r$, then
   $$I^{\left[p^e\right]}\,=\,\left(h_i^{p^e}\; | 1\leq i\leq r\right).$$
For an integer $e\geq 1$, let
   $$\nu(e):=\,\mbox{ largest } i\,\mbox{ such that } \fra^i\nsubseteq\frm_P^{\left[p^e\right]}.$$
Note that since $\fra\subseteq\frm_P$, each $\nu(e)$ is finite. 
Whenever $\fra$ is not understood from the context, we write $\nu_{\fra}(e)$ instead of $\nu(e)$.
By definition,
there exists $h\in\fra^{\nu(e)}\smallsetminus
   \frm_P^{\left[p^e\right]}$.
Since the Frobenius morphism on $R$ is flat, we get $h^p\in \fra^{p\nu(e)}\smallsetminus\frm_P^{[p^{e+1}]}$, hence
$\nu(e+1)\geq p\cdot \nu(e)$. It follows that
$\sup_{e\geq 1}\frac{\nu(e)}{p^e}=\lim_{e\to\infty}\frac{\nu(e)}{p^e}$, and this limit 
is the $F$-\emph{pure threshold} of $\fra$ at $P$, denoted by $\fpt_P(\fra)$. We make the convention that
$\fpt_P(\fra)=\infty$ if $P$ does not lie in $V(\fra)$. 

\subsection{Examples of  computations of $F$-pure thresholds}
We now give some easy examples of $F$-pure thresholds. The reader can compare the resulting values with the corresponding ones for log canonical thresholds in characteristic zero.

\begin{example}\label{smooth_case}
If $\dim(X)=n$, then $\fpt_P(\frm_P)=n$. In fact, for every $e\geq 1$ we have $\nu(e)=(p^e-1)n$.
 Indeed, it is easy to check that if $x_1,\ldots,x_n$ are local coordinates at $P$, then
 $(x_1\cdots x_n)^{p^e-1}\not\in\frm_P^{[p^e]}$, but $\frm_P^{(p^e-1)n+1}\subseteq
 \frm_P^{[p^e]}$. More generally, one can show that if $\fra$ defines a nonsingular subvariety
 of codimension $r$ at $P$,
  then $\nu(e)=r(p^e-1)$ for every $e\geq 1$, hence
 $\fpt_P(\fra)=r$.
\end{example}

\begin{example}
It is a consequence of \cite[Theorem~6.10]{HY} that if $\fra\subset k[x_1,\ldots,x_n]$ is an
ideal generated by monomials, then the $F$-pure threshold is given by the same formula as the log canonical threshold (see Example~\ref{monomial} above for the notation):
$$\fpt_0(\fra)=\max\{\lambda\in\RR_{\geq 0}\mid (1,\ldots,1)\in\lambda\cdot
P(\fra)\}.$$
\end{example}

\begin{example}\label{cusp}
Let $f=x^2+y^3\in k[x,y]$, where $p={\rm char}(k)>3$, and let $P$ be the origin. In order to compute 
$\nu(1)$, we need to find out the largest $r\leq p-1$ with the property that there
are nonnegative $i$ and $j$ with $i+j=r$ such that $2i\leq p-1$ and $3j\leq p-1$.
We conclude that $\nu(1)=\rounddown{\frac{p-1}{2}}+\rounddown{\frac{p-1}{3}}$, hence
$$\nu(1)=\left\{\begin{array}{ccc}
      \frac56(p-1) & \text{if} & p\equiv 1\; ({\rm mod}\, 3)\\
      \frac{5p-7}{6} & \text{if} & p\equiv 2\; ({\rm mod}\, 3).
      \end{array}\right.$$
One can perform similar, but slightly more involved computations in order to get
$\nu(e)$ for every $e\geq 2$, and one concludes (see \cite[Example~4.3]{MTW})
$$\fpt_0(f)=\left\{\begin{array}{ccc}
      \frac56 & if & p\equiv 1\; ({\rm mod}\, 3)\\
      \frac56-\frac{1}{6p} & if & p\equiv 2\; ({\rm mod}\, 3).
      \end{array}\right.$$
Recall that in characteristic zero we have $\lct_0(x^2+y^3)=\frac{1}{2}+\frac{1}{3}=\frac{5}{6}$
(see Example~\ref{diagonal}).
\end{example}

\begin{example}\label{cone_over_elliptic_curve}
Let $f\in k[x,y,z]$ be a homogeneous polynomial of degree $3$, having an isolated singularity
at the origin $P$. Therefore $f$ defines an elliptic curve $C$ in $\PP_k^2$. One can show that
$\fpt_0(f)\leq 1$, with equality if and only if $\nu(1)=p-1$ (see \cite[Example~4.6]{MTW}).
On the other hand, $\nu(1)=p-1$ if and only if $f^{p-1}\not\in (x^p,y^p,z^p)$, which is the case
if and only if the coefficient of $(xyz)^{p-1}$ in $f^{p-1}$ is nonzero. 
This is equivalent to $C$ being an \emph{ordinary} elliptic curve. We refer to
\cite[\S IV.4]{Hartshorne} for this notion, as well as for other equivalent characterizations.
We only mention that $C$ is ordinary if and only if the endomorphism of $H^1(C,\cO_C)$ induced
by the Frobenius morphism is bijective. 
A recent result due to Bhatt \cite{Bhatt} says that if $C$ is not ordinary (that is, $C$ is
supersingular), then $\fpt_0(f)=1-\frac{1}{p}$.
\end{example}

\subsection{Basic properties of the $F$-pure threshold} Part of the interest 
in the $F$-pure threshold comes from the fact that it has similar properties
with the log canonical threshold in characteristic zero. The reader should compare the following
properties to those we discussed in \S 1.3 for the log canonical threshold. An interesting point is that some of the more subtle properties of the log canonical threshold (such as, for example,
Inversion of Adjunction) are straightforward in the present context.

\begin{property}\label{monotonicity2}
If $\fra\subseteq \frb$, then $\fpt_P(\fra)\leq\fpt_P(\frb)$ for every $P\in X$. This is an immediate consequence of the fact that if $\fra^r\not\subseteq\frm_P^{[p^e]}$, then 
$\frb^r\not\subseteq\frm_P^{[p^e]}$, hence $\nu_{\frb}(e)\geq\nu_{\fra}(e)$.
\end{property}

\begin{property}\label{power2}
We have $\fpt_P(\fra^r)=\frac{\fpt_P(\fra)}{r}$. Indeed, it follows easily from definition that
$$r\cdot\nu_{\fra^r}(e)\leq\nu_{\fra}(e)\leq r(\nu_{\fra^r}(e)+1)-1.$$
Dividing by $rp^e$, and letting $e$ go to infinity, gives the assertion.
\end{property}

\begin{property}\label{upper_bound2}
If $\dim(X)=n$, then $\fpt_P(\fra)\leq \frac{n}{\ord_P(\fra)}$. The proof is entirely similar to that of Property~\ref{upper_bound}, using Example~\ref{smooth_case}, and the properties we proved so far.
\end{property}

\begin{property}\label{inversion_of_adjunction2}
The analogue of Inversion of Adjunction holds in this case: if $Y\subset X$ is a nonsingular closed subvariety such that $\fra\cdot\cO_Y$ is nonzero, then $\fpt_P(\fra)\geq
\fpt_P(\fra\cdot\cO_Y)$ for every $P\in Y$. This follows from the fact that
$\fra^i\subseteq\frm_P^{[p^e]}$ implies $(\fra\cO_Y)^i\subseteq(\frm_P\cO_Y)^{[p^e]}$, hence
$\nu_{\fra}(e)\geq\nu_{\fra\cO_Y}(e)$ for every $e\geq 1$. 
\end{property}

\begin{property}\label{lower_bound2}
For every $P\in X$ we have $\fpt_P(\fra)\geq \frac{1}{\ord_P(\fra)}$. This follows as in the
case of Property~\ref{lower_bound}, using Property~\ref{inversion_of_adjunction2} and the fact
that when $\dim(X)=1$, we have $\fpt_P(\fra)=\frac{1}{\ord_P(\fra)}$.
\end{property}

\begin{property}\label{sum_of_ideals2}
If $\fra$ and $\frb$ are nonzero ideals on $X$, then
$$\fpt_P(\fra+\frb)\leq\fpt_P(\fra)+\fpt_P(\frb)$$
for every $P$. Indeed, note that if $\fra^r\subseteq\frm_P^{[p^e]}$ and $\frb^s\subseteq
\frm_P^{[p^e]}$, then $(\fra+\frb)^{r+s}\subseteq\frm_P^{[p^e]}$. Therefore
$$\nu_{\fra+\frb}(e)\leq\nu_{\fra}(e)+\nu_{\frb}(e)+1.$$
Dividing by $p^e$ and taking the limit gives the assertion.
\end{property}

\begin{property}\label{truncation2} 
If $\fra+\frm_P^N=\frb+\frm_P^N$, then
$$|\fpt_P(\fra)-\fpt_P(\frb)|\leq \frac{n}{N},$$
where $n=\dim(X)$. The argument follows the one for Property~\ref{truncation},
using the properties we proved so far.
\end{property}

\subsection{Comparison via reduction mod $p$}

As the above discussion makes clear, there are striking analogies between the log canonical threshold in characteristic zero and the $F$-pure threshold in positive characteristic. Furthermore, as Example~\ref{cusp} illustrates, there are subtle connections between
the log canonical threshold of an ideal and the $F$-pure thresholds of its
reductions mod $p$.  

For simplicity, we will restrict ourselves to the simplest possible setting, as follows. 
Let $\fra\subset\ZZ[x_1,\ldots,x_n]$ be an ideal contained in $(x_1,\ldots,x_n)$. 
On one hand, we consider $\fra\cdot\CC[x_1,\ldots,x_n]$, and with a slight abuse of notation we write
$\lct_0(\fra)$ for the log canonical threshold of this ideal at the origin.

On the other hand, for every prime $p$ we consider the reduction
$\fra_p=\fra\cdot\FF_p[x_1,\ldots,x_n]$ of $\fra$ mod $p$. We correspondingly consider
the $F$-pure threshold at the origin $\fpt_0(\fra_p)$, and the main question is what is the relation between $\lct_0(\fra)$ and $\fpt_0(\fra_p)$ when $p$ varies. 
Example~\ref{cusp} illustrates very well what is known and what is expected in this direction.
The main results in this direction are due to Hara and Yoshida \cite{HY}.

\begin{theorem}\label{thm1_HY}
With the above notation, for $p\gg 0$ we have $\lct_0(\fra)\geq \fpt_0(\fra_p)$.
\end{theorem}

\begin{theorem}\label{thm2_HY}
With the above notation, we have $\lim_{p\to\infty}\fpt_0(\fra_p)=\lct_0(\fra)$.
\end{theorem}

As we will explain in the next subsection, in fact the results of Hara and Yoshida 
concern the relation between the multiplier ideals in characteristic zero and
the so-called test ideals in positive characteristic. The above results are consequences 
of the more general Theorems~\ref{thm3_HY} and \ref{thm4_HY} below.
It is worth mentioning that while
the proof of Theorem~\ref{thm1_HY} above is elementary, that of Theorem~\ref{thm2_HY}
relies on previous work (due independently to Hara \cite{Hara} and Mehta and Srinivas
\cite{MehtaSrinivas}) using the action of the Frobenius morphism on the de Rham complex
and techniques of Deligne-Illusie \cite{DI}.

The main open question in this direction is the following (see \cite[Conjecture~3.6]{MTW}).

\begin{conjecture}\label{conj1}
With the above notation, there is an infinite set $S$ of primes such that
$\lct_0(\fra)=\fpt_0(\fra_p)$ for every $p\in S$.
\end{conjecture}

For example, it was shown in \cite[Example~4.2]{MTW} that if $f=\sum_{i=1}^rc_ix^{\alpha_i}
\in k[x_1,\ldots,x_n]$ is such that the $\alpha_i=
(\alpha_{i,1},\ldots,\alpha_{i,n})\in\ZZ_{\geq 0}^n$ are affinely independent\footnote{This means
that if $\sum_{i=1}^r\lambda_i\alpha_i=0$, with $\lambda_i\in\QQ$
such that $\sum_{i=1}^r\lambda_i=0$, then all $\lambda_i=0$.} and $c_i\in\ZZ$, then there is 
$N$ such that $\lct_0(f)=\fpt_0(f_p)$ whenever
$p\equiv 1$ (mod $N$). 
For example, this applies for the diagonal hypersurface $f=x_1^{a_1}+\ldots+x_n^{a_n}$,
when one can take $N=a_1\cdot\ldots\cdot a_n$. 
We note that the condition $p\equiv 1$ (mod $N$) can be rephrased by saying that $p$ splits completely
in the cyclotomic field generated by the $N^{\rm th}$ roots of $1$.

A particularly interesting case is that of a cone over an elliptic curve. Suppose that
$f\in\ZZ[x,y,z]$ is a homogeneous polynomial of degree 3 such that the corresponding
projective curve $Y\hookrightarrow\PP_{\ZZ}^2$ has the property that
$Y_{\QQ}=Y\times_{\Spec \ZZ}\Spec \QQ$ is nonsingular. We denote by $Y_p$ the corresponding curve in $\PP_{\FF_p}^2$, and we assume that  $p\gg 0$, so that $Y_p$ is nonsingular. Recall that by Example~\ref{homogeneous}, we have $\lct_0(f)=1$, while
Example~\ref{cone_over_elliptic_curve} shows that $\fpt_0(f_p)=1$ if and only if
$Y_p$ is ordinary. The behavior with respect to $p$ depends on whether $Y_{\QQ}$ has complex multiplication. If this is the case, then $Y_p$ is ordinary if and only if $p$
splits in the imaginary quadratic CM field. On the other hand, if $Y_{\QQ}$ does not have complex multiplication, then it is known that the set of primes $p$ such that $Y_p$ is ordinary
has density one \cite{Serre}, but its complement is infinite \cite{Elkies}. 
This shows that unlike the case of Example~\ref{cusp}, the set of primes $p$ such that
$\lct_0(f)=\fpt_0(f_p)$ can be quite complicated. On the other hand, it is known that
in the case of an elliptic curve, there is a number field $K$ such that whenever a prime
$p$ splits completely in $K$, we have $Y_p$ ordinary (see \cite[Exercise~V.5.11]{Silverman}).
It light of these two examples, one can speculate that there is always a number field $L$
such that  if $p$ splits completely in $L$, then $\lct_0(\fra)=\fpt_0(\fra_p)$
(note that by Chebotarev's theorem, this would imply the existence of a set of primes of positive density that satisfies Conjecture~\ref{conj1}). 

We now describe another conjecture that this time has nothing to do with singularities.
If $X\subseteq\PP_{\QQ}^N$ is a projective variety, then we may choose homogeneous 
polynomials
$f_1,\ldots,f_r\in\ZZ[x_0,\ldots,x_N]$ whose images in $\QQ[x_0,\ldots,x_N]$ generate the ideal 
of $X$. For a prime $p$, we get a projective variety $X_p\subseteq\PP_{\FF_p}^N$
defined by the ideal generated by the images of $f_1,\ldots,f_r$ in $\FF_p[x_0,\ldots,x_n]$.
Given another choice of such $f_1,\ldots,f_r$, the varieties $X_p$ are the same for $p\gg 0$.
Note that if $X$ is smooth and geometrically connected\footnote{Recall that this means
that $X\times_{\Spec\QQ}\Spec\overline{\QQ}$ is connected}, then for every 
$p\gg 0$, the variety $X_p$ is again smooth and geometrically connected.
Similar considerations can be made when starting with a variety defined over a number field.

\begin{conjecture}\label{conj21}
If $X$ is a smooth, geometrically connected, $n$-dimensional projective variety 
over $\QQ$, then there are infinitely many primes $p$ such that the endomorphism induced
by the Frobenius on $H^n(X_p,\cO_{X_p})$ is bijective. More generally, a similar assertion holds if
$X$ is defined over an arbitrary number field.
\end{conjecture}

We mention that this conjecture is open even in the case when $X$ is a curve of genus $\geq 3$. As the following result from \cite{MuS} shows, this conjecture implies the expected
relation between log canonical thresholds and $F$-pure thresholds.

\begin{theorem}
If Conjecture~\ref{conj21} is true, then so is Conjecture~\ref{conj1}.
\end{theorem}

\subsection{Test ideals and $F$-jumping numbers}
As we have seen in \S 2.2, the behavior of the $F$-pure threshold is similar to that of the log canonical threshold.
There is however one property of the log canonical threshold that is more subtle in the case of the $F$-pure threshold, namely its rationality. In order to prove this for $F$-pure thresholds, one has to involve also the ``higher jumping numbers". 

In this section we give a brief introduction to test ideals. In the same way that the
$F$-pure threshold is an analogue of the log canonical threshold in positive characteristic,
the test ideals give an analogue of the multiplier ideals in the same context. They have been
defined by Hara and Yoshida \cite{HY} for rather general ambient varieties, and it was shown that they behave in the same way as the multiplier ideals do in characteristic zero.
 Their definition involved a generalization of the theory of tight closure of \cite{HH} to the case where instead of dealing
with just one ring, one deals with a pair $(R,\fra^{\lambda})$, where $\fra$ is an ideal in 
$R$, and $\lambda\in\RR_{\geq 0}$. In the case of an ambient nonsingular variety, it was shown
in \cite{BMS} that one can give a more elementary definition. This is the approach that we are going to take. We will just sketch the proofs, and refer to
\cite{BMS} for details. For a survey of test ideals in the general setting, see \cite{ST}.

Given any ideal $\frb\subseteq R$ and $e\geq 1$, we claim that there is a unique smallest ideal $J$ such that $\frb\subseteq J^{[p^e]}$. Indeed, if $(J_i)_i$ is a family of ideals such that
$\frb\subseteq J_i^{[p^e]}$, then $\frb\subseteq \bigcap_i J_i^{[p^e]}=\left(\bigcap_iJ_i\right)^{[p^e]}$ (the equality follows from the fact that $R$ is a projective module via the Frobenius morphism). We denote the ideal $J$ as above by $\frb^{[1/p^e]}$. 

Given a nonzero ideal $\fra$ in $R$ and $\lambda\in\RR_{\geq 0}$, we consider for every
$e\geq 1$ the ideal
$I_e:=(\fra^{\lceil\lambda p^e\rceil})^{[1/p^e]}$. It is easy to see using the minimality in the definition of the ideals $\frb^{[1/p^e]}$ that we have $I_e\subseteq I_{e+1}$ for every $e\geq 1$.
Since $R$ is Noetherian, these ideals stabilize for $e\gg 0$ to the \emph{test ideal}
$\tau(\fra^{\lambda})$. It is not hard to check that this definition commutes with inverting a nonzero element 
in $R$, hence we get in this way coherent sheaves on $X$. 

In many respects, the test ideals satisfy the same formal properties that the multiplier ideals
do in characteristic zero. It is clear from definition that if $\lambda\leq\mu$, then
$\tau(\fra^{\mu})\subseteq\tau(\fra^{\lambda})$. While it requires a little argument, it is
elementary to see that given any $\lambda$, there is $\epsilon>0$ such that
$\tau(\fra^t)=\tau(\fra^{\lambda})$ for every $t$ with $\lambda\leq t\leq\lambda+\epsilon$.
By analogy with the case of multiplier ideals, one says that $\lambda$ is an $F$-jumping
number of $\fra$ if $\tau(\fra^{\lambda})\neq\tau(\fra^t)$ for every $t<\lambda$. 

It is easy to see from definition that in the case of a principal ideal we have
$\tau(f^{\lambda})=f\cdot\tau(f^{\lambda-1})$ for every $\lambda\geq 1$. 
Furthermore, we also have an analogue of Skoda's theorem: if $\fra$
is generated by $m$ elements, then $\tau(\fra^{\lambda})=\fra\cdot\tau(\fra^{\lambda-1})$
for every $\lambda\geq m$. It is worth pointing out that the proof in this case
(see \cite[Proposition~2.25]{BMS}) is 
much more elementary than in the case of multiplier ideals.

Note that if $P$ is a closed point on $X$ defined by the ideal $\frm_P$, then
$\fpt_P(\fra)$ is the smallest $\lambda$ such that $\tau(\fra^{\lambda})\not\subseteq 
\frm_P$. Indeed, by definition the latter condition is equivalent with the existence of
an $e\geq 1$ such that $\left(\fra^{\lceil\lambda p^e\rceil}\right)^{[1/p^e]}$ is not contained in
$\frm_P$, which in turn is equivalent to $\fra^{\lceil \lambda p^e\rceil}\not\subseteq
\frm_P^{[p^e]}$. We can further rewrite this as $\nu_{\fra}(e)\geq\lceil \lambda p^e\rceil$. 
Since $\fpt_P(\fra)=\sup_{e'}\frac{\nu_{\fra}(e')}{p^{e'}}$, it follows that if 
$\fpt_P(\fra)>\lambda$, then there is 
$e$ such that $\frac{\nu_{\fra}(e)}{p^e}> \lambda$, hence $\nu_{\fra}(e)\geq\lceil
\lambda p^e\rceil$. Conversely, if $\tau(\fra^{\lambda})\not\subseteq\frm_P$ and if
we take $\epsilon>0$ such that $\tau(\fra^{\lambda})=\tau(\fra^{\lambda+\epsilon})$,
then the above discussion implies that there is $e$ such that
$\nu_{\fra}(e)\geq \lceil (\lambda+\epsilon)p^e\rceil$, hence
$$\fpt_P(\fra)\geq\frac{\nu_{\fra}(e)}{p^e}\geq \frac{\lceil(\lambda+\epsilon)p^e\rceil}{p^e}
\geq \lambda+\epsilon>\lambda.$$

The global $F$-pure threshold $\fpt(\fra)$ is the smallest $F$-jumping number, that is, the smallest
$\lambda$ such that $\tau(\fra^{\lambda})\neq R$. It is clear from the above discussion that
$\fpt(\fra)=\min_{P\in X}\fpt_P(\fra)$ and $\fpt_P(\fra)=\max_U\fpt(\fra\vert_U)$, where
$U$ varies over the affine open neighborhoods of $P$. 
The following result from \cite{BMS} gives the analogue for the rationality and the discreteness
of the jumping numbers of the multiplier ideals of a given ideal. For extensions to various
other settings, see \cite{BMS0}, \cite{KLZ} and \cite{BSTZ}.

\begin{theorem}
If $\fra$ is a nonzero ideal in $R$, then the set of $F$-jumping numbers of $\fra$
is a discrete set of rational numbers.
\end{theorem}

\begin{proof}[Sketch of proof]
The new phenomenon in positive characteristic is that for every $\lambda$, we have
\begin{equation}\label{eq1_test_ideals}
\tau(\fra^{\lambda/p})=\tau(\fra^{\lambda})^{[1/p]}.
\end{equation}
This follows from the fact that for $e\gg 0$ we have
$$\tau(\fra^{\lambda/p})=\left(\fra^{\lceil \lambda p^e\rceil}\right)^{[1/p^{e+1}]}=
\left(\left(\fra^{\lceil \lambda p^e\rceil}\right)^{[1/p^{e}]}\right)^{[1/p]}=
\tau(\fra^{\lambda})^{[1/p]}.$$
It is an immediate consequence of (\ref{eq1_test_ideals}) that
if $\lambda$ is an $F$-jumping number of $\fra$, then also $p\lambda$ is an $F$-jumping
number. 

The second ingredient in the proof of the theorem is given by a bound
on the degrees of the generators of $\tau(\fra^{\lambda})$ in terms of the degrees of the generators of $\fra$, in the case when $R=k[x_1,\ldots,x_n]$. One shows that in general,
if $\frb\subseteq k[x_1,\ldots,x_n]$ is an ideal generated in degree $\leq d$, then
$\frb^{[1/p^e]}$ is generated in degree $\leq d/p^e$. This is a consequence
of the following description of $\frb^{[1/p^e]}$. Consider the basis of $R$ over
$R^{p^e}=k[x_1^{p^e},\ldots,x_n^{p^e}]$ given by the monomials $w_1,\ldots,w_{np^e}$ of degree
$\leq p^e-1$ in each variable. If $\frb$ is generated by $h_1,\ldots,h_m$, and if we write
$$h_i=\sum_{j=1}^{np^e}u_{i,j}^{p^e}w_j,$$
then $\frb^{[1/p^e]}=(u_{i,j}\mid i\leq m, j\leq np^e)$. This follows from definition and the fact
that $h_i\in J^{[p^e]}$ if and only if $u_{i,j}\in J$ for all $j$. 

Suppose now that $\fra$ is an ideal in $k[x_1,\ldots,x_n]$ generated in degree $\leq d$.
Since $\tau(\fra^{\lambda})=\left(\fra^{\lceil {\lambda p^e}\rceil}\right)^{[1/p^e]}$ for all
$e\gg 0$, we deduce that $\tau(\fra^{\lambda})$
is generated in degree $\leq \lambda d$. This implies that there are only finitely many $F$-jumping numbers of $\fra$ in $[0,\lambda]$. Indeed, otherwise we would get
an infinite decreasing sequence of linear subspaces of the vector space of polynomials
in $x_1,\ldots,x_n$ of degree $\leq\lambda d$. 

We thus obtain the discreteness of the $F$-jumping numbers in the case of a polynomial ring.
The rationality follows easily. If $\fra$ is principal and $\lambda$ is an  $F$-jumping number,
then so are $p\lambda$ and $\lambda-1$ (assusming $\lambda>1$). It follows that
for every $\lambda$, the fractional part of $p^e\lambda$ is an $F$-jumping number for every $e\geq 1$. Since we have only finitely many such numbers in $[0,1]$, we conclude that
$\lambda\in\QQ$. The case of an arbitrary ideal is proved similarly, using the analogue of
Skoda's theorem. The case of an arbitrary regular ring $R$ of finite type over $k$ can be
then reduced to that of a polynomial ring.
\end{proof}

An interesting feature of the analogy between test ideals and multiplier ideals is that
some of the more subtle properties of multiplier ideals, whose proofs involve
vanishing theorems (such as the Restriction Theorem, the Subadditivity Theorem and the Skoda Theorem) follow directly from definition in the case of test ideals. On the other hand,  
some properties of multiplier ideals that are simple consequences of the description in terms of resolution (for example, the fact that such ideals are integrally closed) can fail for
test ideals. For this and related facts, see \cite{MY}.

The results that we mentioned relating the log canonical threshold and the $F$-pure threshold
via reduction mod $p$ have a stronger form relating the multiplier ideals and 
the test ideals. The following two results have been proved\footnote{Actually, the results in \emph{loc. cit.} are in the context of local rings. However, using the arguments therein, one can get the global version of these results that we give.} by Hara and Yoshida in
\cite{HY}. Note that they imply Theorems~\ref{thm1_HY} and \ref{thm2_HY} above.
We keep the notation in these two theorems. 
Using the description of the multiplier ideals in terms of a log resolution, one can show that
all multiplier ideals of $\fra\cdot\CC[x_1,\ldots,x_n]$ are obtained by base-extension
from ideals in the ring $\ZZ[1/N][x_1,\ldots,x_n]$ for some positive integer $N$. In particular,
for every $p>N$ we may define the reductions mod $p$ of the multiplier ideals, that we denote by $\cJ(\fra^{\lambda})_p$.

\begin{theorem}\label{thm3_HY}
If $p\gg 0$, then $\tau(\fra_p^{\lambda})\subseteq \cJ(\fra^{\lambda})_p$ for every $\lambda$.
\end{theorem}

\begin{theorem}\label{thm4_HY}
For every $\lambda\in\RR_{\geq 0}$, we have $\tau(\fra_p^{\lambda})= \cJ(\fra^{\lambda})_p$
for all $p$ large enough (depending on $\lambda$). 
\end{theorem}

The following is a stronger version of Conjecture~\ref{conj1}.

\begin{conjecture}\label{conj3}
Given $\fra$, there is an infinite set of primes $S$ such that 
$\tau(\fra_p^{\lambda})= \cJ(\fra^{\lambda})_p$ for every $\lambda\in\RR_{\geq 0}$ and every 
$p\in S$.
\end{conjecture}

The result in \cite{MuS} that we have already mentioned says that, in fact, 
Conjecture~\ref{conj21} implies Conjecture~\ref{conj3}. On the other hand, it is shown in 
\cite{Mustata1} that a slightly more general version of Conjecture~\ref{conj3}
(that deals with ideals in $\overline{\QQ}[x_1,\ldots,x_n]$)
implies Conjecture~\ref{conj21}. Therefore the conjecture relating the multiplier ideals
to the test ideals via reduction mod $p$ is equivalent to the conjecture concerning the Frobenius action on the reductions to positive characteristic of a smooth projective variety.

\section{Shokurov's ACC conjecture}

In this section we turn to Shokurov's ACC conjecture for log canonical thresholds
from \cite{Sho}. This has been
proven in the case of ambient smooth varieties in \cite{dFEM}, building on work from \cite{dFM} and \cite{Kol08}.
Recall that a set satisfies the ascending chain condition (ACC, for short) if
it contains no infinite strictly increasing sequences.

\begin{theorem}\label{ACC}
For every $n$, the set
${\mathcal T}_n$ of all log canonical thresholds $\lct_P(\fra)$,
where $\fra$ is a nonzero ideal on an $n$-dimensional nonsingular complex algebraic variety
$X$ and $P\in V(\fra)$, satisfies ACC.
\end{theorem}

\begin{remark}\label{general_form}
As we have already mentioned in Introduction, Shokurov's conjecture is formulated when the ambient variety is not necessarily smooth, but only has klt singularities, and in fact more generally, when one deals with a pair $(X,D)$ with klt singularities, where $D$ is an effective 
$\QQ$-divisor on $X$, with a suitable condition on the coefficients. We refer to
\cite{Birkar} for the precise statement\footnote{In a very recent breakthrough, a proof of the general version of the ACC conjecture was announced in \cite{HMX}.
That proof goes far beyond the scope of these notes,
 relying heavily on techniques from the
Minimal Model Program.}.
\end{remark}

The interest in this conjecture (aside from its intrinsic appeal) comes from the connections with one of the remaining open problems in the Minimal Model Program. As an aside, let us mention
that after the recent breakthrough in \cite{BCHM}, there are two fundamental remaining open problems in this program:
  \begin{itemize}
   \item Termination of Flips (proved for certain sequences of flips in the case of varieties of general type in \cite{BCHM}).
   \item Abundance, that is $K_X$ nef implies $K_X$ semiample.
   \end{itemize}

Via work of  Birkar \cite{Birkar}, the ACC conjecture is related to Termination of Flips, as follows.
Suppose that Termination of Flips is known in dimension $n$, and that Shokurov's
ACC conjecture (in its general form mentioned in Remark~\ref{general_form}) is known in dimension $n+1$, then Termination of Flips follows in dimension $n+1$ for sequences of flips
of pairs $(X,D)$ such that $K_X+D$ is numerically equivalent to an effective
$\QQ$-divisor. While
this is the most interesting case (this is when one expects at the end of the Minimal Model Program to get a minimal model), this extra condition on $(X,D)$ which does not appear in the inductive hypothesis, does not allow to deduce in general
Termination of Flips from the ACC conjecture. 

The general case of Shokurov's conjecture in known in dimensions 2 and 3, by work
of Shokurov \cite{Sho} and Alexeev \cite{Alexeev}. The methods used to prove Theorem~\ref{ACC} above, also allow to prove the same result under weaker assumptions on the singularities of the ambient variety:
 \begin{itemize}
   \item $X$ with quotient singularities, see \cite{dFEM}.
   \item $X$ with locally complete intersection singularities, see \cite{dFEM}. The key point is that Inversion  of Adjunction works well in this setting.
   \item $(X,P)$ with ''bounded singularities'', in the sense that one assumes that $\widehat{\cO_{X,P}}$ is isomorphic to some $\widehat{\cO_{Y,P}}$, where $Y$ is defined in a fixed $\AAA^N$ by equations of bounded degree, see \cite{dFEM1}. Note that this bounds the embedding dimension of $(X,P)$, and this is a key obstruction towards proving the general case
   of Shokurov's conjecture by these methods.
   \end{itemize}
   
We do not give the proof of Theorem~\ref{ACC}, but explain instead an idea that goes into the proof. We show how this is used in order to prove the following result from
\cite{dFM} and \cite{Kol08}.

\begin{theorem}\label{case0}
For every $n\geq 1$, the set
${\mathcal T}^{\rm div}_n$ of all log canonical thresholds $\lct_P(f)$, where $P$ is a point
on an $n$-dimensional nonsingular complex algebraic variety, and $f\in\cO(X)$ vanishes at $P$,
 is a closed subset of $\RR$.
\end{theorem}
   
There are two important points concerning the proofs of Theorem~\ref{case0}.
First, it is convenient to work with log canonical thresholds of formal power series
$f\in k\llbracket x_1,\ldots,x_n\rrbracket$, where $k$ is an arbitrary field   
of characteristic zero. The basic properties of log canonical thresholds that we discussed
extend to this setting, see \cite{dFM}. 
The key point is that results of \cite{Temkin} provide existence of log resolutions in this setting.

A second idea is that given a sequence of polynomials $(f_m)_{m\geq 1}$ in 
$\CC[x_1,\ldots,x_n]$ such that $\lim_{m\to \infty}\lct_0(f_m)=\alpha$, there is 
$F\in K\llbracket x_1,\ldots,x_n\rrbracket$ such that $\lct(F)=\alpha$, for some 
algebraically closed field
$K$ containing $\CC$. Once this is done, an easy argument shows that there is
a polynomial
$f\in \CC[x_1,\ldots,x_n]$ such that $\lct_0(f)=\lct(F)$.
The construction of $F$ can be achieved in two ways: using ultrafilters (as in 
\cite{dFM}) or using an infinite sequence of generic points (as in \cite{Kol08}). 
In what follows we discuss the former method. 
Let us begin by reviewing the definition of ultrafilters.
   
   \begin{definition}
   A \emph{filter} on $\NN=\ZZ_{>0}$ is a collection ${\mathcal U}$ of subsets of $\NN$
   such that
   \begin{itemize}
   \item[1)] $\emptyset\not\in{\mathcal U}$;
   \item[2)] $A,B\in{\mathcal U}\,\Longrightarrow\, A\cap B\in{\mathcal U}$;
   \item[3)] $A\in{\mathcal U},\,B\supseteq A\,\Longrightarrow B\in {\mathcal U}$.
   \end{itemize}
   A filter is called an \emph{ultrafilter} if it is maximal, in the sense that it is not properly contained in another filter. Equivalently,
   \begin{itemize}
   \item[4)] For every $A\subseteq\NN$, either $A$ or its complement $\NN\setminus A$ is in 
   ${\mathcal U}$.
   \end{itemize}
   An ultrafilter is called \emph{principal} if there is $a\in\NN$ that is contained in all
   $A\in {\mathcal U}$ (in which case, by maximality, we have
   ${\mathcal U}=\{A\subseteq\NN\mid a\in A\}$). 
   \end{definition}   

It is easy to see that an ultrafilter ${\mathcal U}$ is non-principal if and only if  the complement of every finite proper subset of $\NN$  is in ${\mathcal U}$. One can show  using the Kuratowski-Zorn Lemma that there are ultrafilters containing the filter $\{\NN\smallsetminus A\mid A\subseteq\NN\,\text{finite}\}$,
hence there are non-principal ultrafilters
\footnote{In fact, the existence of non-principal ultrafilters is equivalent to the Kuratowski-Zorn Lemma, hence to the axiom of choice.}. Let us fix such a non-principal ultrafilter ${\mathcal U}$. 

Given a set $A$, its \emph{non-standard extension} is
$$\* A:= A^{\NN}/\sim$$
 where the equivalence relation on $A^{\NN}$ 
 is defined by
 $$(a_m)\sim (b_m)\,\,\text{if}\,\,\{m\in\NN\mid a_m=b_m\}\in {\mathcal U}$$
 (in this case, one also says that $a_m=b_m$ \emph{for almost all} $m$).
 The class of a sequence $(a_m)$ in $\*  A$ is denoted by $[a_m]$.
 There is an injective map $A\hookrightarrow \* A$ that takes $a\in A$ to $[a]$
(the class of the constant sequence). 

The principle is that whatever algebraic structure $A$ has, this extends to $\* A$. 
For example, if $k$ is a field, then $\* k$ is a field, with addition and multiplication
defined by 
$$[a_m]+[b_m]=[a_m+b_m]\,\,\text{and}\,\,[a_m]\cdot [b_m]=[a_m\cdot b_m].$$
Let us see, for example, that every nonzero element in $\* k$ has an inverse
(of course, the zero element in $\* k$ is the image of the zero element in $k$):
if $[a_m]\neq 0$, then the set $T=\{m\mid a_m\neq 0\}$ lies in ${\mathcal U}$. If
we put $b_m=a_m^{-1}$ for $m\in T$, and $b_m\in k$ arbitrary for $m\not\in T$, then
$[a_m]\cdot [b_m]=1$.

Recall that for $u=(u_1,\ldots,u_n)\in\ZZ_{\geq 0}^n$, we put $x^u=x_1^{u_1}\cdots x_n^{u_n}$
and $|u|=\sum_iu_i$.
We may identify $(\* k)[x_1,\ldots,x_n]$ with the set of those $[f_m]\in \*(k[x_1,\ldots,x_n])$
such that there is an integer $d$ with ${\rm deg}(f_m)\leq d$ for all $d$. 
Indeed, given $[f_m]\in \*(k[x_1,\ldots,x_n])$, with $f_m=
\sum_{u\in\ZZ_{\geq 0}^n, |u|\leq d}a_{u,m}x^u$, 
the corresponding polynomial in $f\in (\*k)[x_1,\ldots,x_n]$ is
$\sum_{u\in\ZZ_{\geq 0}^n, |u|\leq d}[a_{u,m}]x^u$. Therefore we write $f=[f_m]$ (note that this is compatible with 
our previous convention). However, a general element in $\*(k[x_1,\ldots,x_n])$
is not a polynomial in $(\*k)[x_1,\ldots,x_n]$.

If $f=[f_m]\in (\*k)[x_1,\ldots,x_n]$, and $a=[a_m]\in \*k$, then $f(a)=[f_m(a_m)]$.
In particular, we have $f(a)=0$ if and only if $f_m(a_m)=0$ for almost all $m$. 
It is then easy to see that if $k$ is algebraically closed, then $\* k$ is algebraically closed,
as well.

Suppose now that $f_m\in\CC[x_1,\ldots,x_n]$ are such that $f_m(0)=0$ for all $m$,
and $\lim_{m\to\infty}\lct_0(f_m)=\alpha$. We may consider $[f_m]\in \*(\CC[x_1,\ldots,x_n])$.
While this is not in general a polynomial, it determines a formal power series $F$
with coefficients in $\*\CC$: if $f_m=\sum_{u\in\ZZ_{\geq 0}^n}a_{m,u}x^u$ for all $m$, then
$F=\sum_{u\in\ZZ_{\geq 0}^n}[a_{m,u}]x^u\in (\* \CC)\llbracket x_1,\ldots,x_n\rrbracket$.

\noindent{\bf Claim}. We have $\lct(F)=\alpha$. 

Given any polynomial or power series $h$, let us denote by $h_{\leq N}$
the truncation of $h$ of degree $\leq N$. It is enough to show that for every $N$,
we have
\begin{equation}\label{eq1_claim}
\lct_0(F_{\leq N})=\lct_0((f_m)_{\leq N})
\end{equation}
for almost all $m$ (hence this holds, in particular, for an infinite set of values of $m$).
Indeed, it follows from an extension of Property~\ref{truncation} 
to power series that
$$|\lct(F)-\lct_0(F_{\leq N})|\leq \frac{n}{N+1}\,\,\text{and}\,\,|\lct_0(f_m)-\lct_0((f_m)_{\leq N})|
\leq \frac{n}{N+1}.$$
Since $|\lct_0(f_m)-\alpha|\leq\frac{n}{N+1}$ for all $m\gg 0$, we deduce from
(\ref{eq1_claim}) that
$|\lct(F)-\alpha|\leq\frac{3n}{N+1}$, and this happens for all $N$, hence the claim.

Note that $F_{\leq N}$ is the polynomial in $(\*\CC)[x_1,\ldots,x_n]$ corresponding to
the sequence $((f_m)_{\leq d})$. After replacing each $f_m$ by $(f_m)_{\leq d}$ we may assume 
that ${\rm deg}(f_m)\leq d$ for every $m$, so that $F$ is a polynomial in
$(\*\CC)[x_1,\ldots,x_n]$ of degree $\leq d$. 

If we parametrize polynomials in $n$ variables,
 of degree $\leq d$ and vanishing at the origin, by their 
coefficients, we find a polynomial ring $R=[y_1,\ldots,y_r]$ (with $r={{n+d}\choose d}-1$),
and a polynomial $h\in R[x_1,\ldots,x_n]$ with $h(0)=0$, such that every polynomial 
in $\CC[x_1,\ldots,x_n]$ corresponds to $h_t$ for a unique closed point of
$\Spec R$. Furthermore, every polynomial in $(\* \CC)[x_1,\ldots,x_n]$ of degree $\leq d$
and vanishing at the origin corresponds to a closed point of $\Spec (R\otimes_{\CC}\*\CC)$.
Using Property~\ref{families1}, we obtain a disjoint decomposition of
$\Spec R$ in locally closed subsets $Z_1,\ldots,Z_d$, and $\alpha_1,\ldots,\alpha_d$ such that for every closed point $t\in Z_i$, we have
$\lct_0(h_t)=\alpha_i$. 

This gives a decomposition of $\NN$ according to which $Z_i$ contains the point corresponding to $f_m$. Since ${\mathcal U}$ is an ultrafilter, it follows that there is $i$ such that
$f_m\in Z_i$ for almost all $m$. The condition for a polynomial $f_m$ to belong
to some $Z_j$ is given by finitely many polynomial expressions in the coefficients of $f_m$
being zero or nonzero. We thus conclude that since $f_m\in Z_i$ for almost all $m$, then
$F$ belongs to $Z_i\times_{\Spec \CC}\Spec \*\CC$, and by construction of the $Z_i$, this implies that $\lct_0(F)=\alpha_i$. This completes the proof of the claim. 

In the above discussion, we started with a sequence of polynomials $(f_m)$ and obtained a formal power series $F$. If we start, more generally,  with  a sequence of ideals $(\fra_m)$
in $\CC[x_1,\ldots,x_n]$ vanishing at $0$, one obtains an ideal ${\mathfrak A}\subseteq
(\*\CC)\llbracket x_1,\ldots,x_n\rrbracket$ contained in the maximal ideal. A similar argument to the one given above
can be used to show that if $\lim_{m\to\infty}\lct_0(\fra_m)=\alpha$, then 
$\lct({\mathfrak A})=\alpha$. For details, we refer to \cite{dFM}.

We can now sketch the proof of Theorem~\ref{case0}.
Note first that by Example~\ref{bounded_by_1}, we have
${\mathcal T}^{\rm div}_n\subseteq [0,1]$.
One can show using Property~\ref{truncation} that if $X$ is an $n$-dimensional
nonsingular  variety and  $f\in \cO(X)$ vanishes at some $P\in X$,
we may write $\lct_P(f)$ as the limit of a sequence $(\lct_0(h_m))_{m\geq 1}$,
for suitable $h_m\in \CC[x_1,\ldots,x_n]$ vanishing at $0$.
Therefore in order to prove Theorem~\ref{case0}, it is enough to show that 
if $f_m\in \CC[x_1,\ldots,x_n]$ are polynomials vanishing at $0$, with
$\alpha=\lim_{m\to\infty}\lct_0(f_m)$, then there is another such polynomial $f$ 
with $\alpha=\lct_0(f)$.

The above construction gives a formal power series
 $F\in(\*\CC)\llbracket x_1,\ldots,x_n\rrbracket$
with $\lct(F)=\alpha$. Let $E$ be a divisor over 
$\Spec \left((\*\CC)\llbracket x_1,\ldots,x_n\rrbracket\right)$ that computes $\lct(F)$,
and let $\xi$ be the generic point of the center of $E$. 
Note that the completion $\widehat{\cO_{X,\xi}}$ is isomorphic to a formal power series ring
$k(\xi)\llbracket x_1,\ldots,x_s\rrbracket$, for some $s\leq n$. If $k=\overline{k(\xi)}$
is an algebraic closure of $k(\xi)$, then we may replace 
$F\in (\*\CC)\llbracket x_1,\ldots,x_n\rrbracket$ by its image $G$ in
$k\llbracket x_1,\ldots,x_n\rrbracket$, and $\lct(G)=\lct(F)=\alpha$.
The advantage is that we now have a divisor $\widetilde{E}$ over 
$\Spec (k\llbracket x_1,\ldots,x_n\rrbracket)$ that computes $\lct(G)$, and whose center is equal to the closed point. In this case
$\lct((G)+\frm^{\ell})=\lct(G)$ for $\ell\gg 0$, where $\frm$ is the ideal defining the closed point.
Indeed, if $\ell>\ord_{\widetilde{E}}(G)$, then $\ord_{\widetilde{E}}((G)+\frm^{\ell})=
\ord_{\widetilde{E}}(G)$, which implies
$$\lct((G)+\frm^{\ell})\leq\frac{{\rm Logdisc}(\widetilde{E})}{\ord_{\widetilde{E}}((G)+\frm^{\ell})}=
\frac{{\rm Logdisc}(\widetilde{E})}{\ord_{\widetilde{E}}(G)}=\lct(G),$$
while the inequality $\lct(G)\leq\lct((G)+\frm^{\ell})$ is a consequence of 
Property~\ref{monotonicity}. The ideal $(G)+\frm^{\ell}$ is the image of an ideal $\frb\subset
k[x_1,\ldots,x_n]$ vanishing at zero, hence $\alpha=\lct((G)+\frm^{\ell})=\lct_0(\frb)$.
Since $\alpha\leq 1$, it follows from Example~\ref{general_coefficients} that
if $g$ is a linear combination of the generators of $\frb$ with general coefficients in $k$, then
$\lct_0(g)=\alpha$. Let $d={\rm deg}(g)$. As we have seen before, we have a disjoint decomposition
$Z_1\sqcup\ldots \sqcup Z_r$ 
of the space 
parametrizing complex polynomials in $n$ variables of degree $\leq d$,
such that points of each $Z_i$ have constant log canonical threshold. If
$g$ corresponds to a point in $Z_i\times_{\Spec\CC}\Spec k$, we see that
a polynomial $f\in\CC[x_1,\ldots,x_n]$ corresponding to a point in $Z_i$ has
$\lct_0(f)=\alpha$. This complets the (sketch of) proof of Theorem~\ref{case0}.

For simplicity, in the above we have restricted the discussion to the case of principal ideals. Minor modifications of the argument allow to prove that the set ${\mathcal T}_n$ in Theorem~\ref{ACC} is
closed in $\RR$. Furthermore, the same circle of ideals allow the proof of the following statement, conjectured by Koll\'{a}r, concerning \emph{decreasing} sequences
of log canonical thresholds.

\begin{theorem}\label{kollar_conjecture}
With the notation in Theorem~\ref{ACC}, the limit of every strictly decreasing sequence
of elements in ${\mathcal T}^{\rm div}_n$ is in ${\mathcal T}_{n-1}^{\rm div}$.
\end{theorem}

The key point is to show (using the notation used for the proof of Theorem~\ref{case0}
above) that if $E$ is a divisor computing the log canonical threshold of
$F\in (\*\CC)\llbracket x_1,\ldots,x_n\rrbracket$, then the center of $E$ is not equal to the closed point. 
In this case, after localizing at the generic point of this center, we end up in a ring of power series in at most $(n-1)$-variables. 

The proof of Theorem~\ref{ACC} is more involved. In addition to the ideas used above,
one has to use the following ingredient.

\begin{theorem}\label{extra_ingredient}
Let $\fra$ be an ideal on a smooth complex variety $X$, and $E$ a divisor over $X$ that computes $\lct_P(\fra)$, for some $P\in X$. If $E$ has center equal to $P$ on $X$, then for every ideal
$\frb$ on $X$ such that $\fra+\frm_P^{\ell}=\frb+\frm_P^{\ell}$, where $\frm_P$ is the ideal of $P$
and $\ell>\ord_E(\fra)$, we have $\lct_P(\frb)=\lct_P(\fra)$.
\end{theorem}

A proof of this result was given in \cite{Kol08} using the results in the Minimal Model Program
from \cite{BCHM}. A more elementary proof, only relying on the Connectedness Theorem
of Shokurov and Koll\'{a}r, was given in \cite{dFEM}.

\section{Asymptotic log canonical thresholds}

In this section we discuss following \cite{JM} an asymptotic version of the log canonical threshold, in the context
of graded sequences of ideals. In particular, we explain a question concerning the computation
of asymptotic log canonical thresholds by quasi-monomial valuations. For proofs and
details we refer to \cite{JM}.

\subsection{Definition and basic properties}
Let $X$ be a smooth, connected, complex algebraic variety. A \emph{graded sequence of ideals} $\fra_{\bullet}$ on $X$ is a sequence $(\fra_m)_{m\geq 1}$ of ideals that satisfies
$$\fra_p\cdot\fra_q\subseteq\fra_{p+q}$$
for every $p,q\geq 1$. All our graded sequences are assumed to be nonzero,
that is, some $\fra_p$ is nonzero.
A trivial example of such a sequence if given by $\fra_m=I^m$, where $I$ is a fixed nonzero ideal on $X$. 

The most interesting example is related to asymptotic base loci of line bundles.
Suppose that $X$ is projective and $L$ is a line bundle on $X$ such that $h^0(X,L^m)\neq 0$
for some $m\geq 1$. If we take $\fra_p$ to be the ideal defining the base locus of the complete
linear series $|L^p|$, then $\fra_{\bullet}$ is a graded sequence of ideals. For other examples
of graded sequences we refer to \cite[Chapter~11.1]{Lazarsfeld}.

We note that if the graded $\cO_X$-algebra $\cO_X\oplus(\bigoplus_{m\geq 1}\fra_m)$
is finitely generated\footnote{For example, if $\fra_m$ defines the base locus of
$L^m$, as above, this condition holds if the section $\CC$-algebra $\oplus_{m\geq 0}\Gamma(X,L^m)$
is finitely generated.}, then there is $p\geq 1$ such that $\fra_{mp}=\fra_p^m$
for every $m\geq 1$ (see  \cite[Chap. III, \S 1, Prop. 2]{Bourbaki}). In this case, we consider
the graded sequence as essentially trivial. The interest in the study of graded sequences
and of their asymptotic invariants arises precisely when this algebra is not finitely generated 
(or at least, when this finiteness is not known a priori). 

Since we are interested in the behavior of singularities, we may, as before,
assume that $X=\Spec R$ is affine. We denote by ${\rm Val}_X$ the space of 
real valuations
of the fraction field of $R$ that are nonnegative on $R$. For example, if $E$ is a divisor 
over $X$, then all positive multiplies of $\ord_E$ lie in ${\rm Val}_X$.

Given a graded sequence of ideals $\fra_{\bullet}$, one can extend ``asymptotically"
usual invariants of ideals, to obtain invariants for the sequence. 
More precisely, suppose that $\alpha(-)$ is an invariant of ideals that satisfies the following two conditions: 
\begin{enumerate}
\item[1)] If $\fra\subseteq\frb$, then $\alpha(\fra)\geq\alpha(\frb)$.
\item[2)] $\alpha(\fra\cdot\frb)\leq\alpha(\fra)+\alpha(\frb)$.
\end{enumerate}
Examples of such an invariants are given by $\alpha(\fra)=v(\fra):=\min\{v(f)\mid f\in\fra\}$,
where $v\in {\rm Val}_X$.
Another example is given by $\alpha(\fra)=\Arn(\fra)$ (the fact that $\Arn$ satisfies 1) and 2)
above follows from Properties~\ref{monotonicity} and \ref{convexity}). 

Given a graded sequence of ideals $\fra_{\bullet}$ and an invariant $\alpha$
as above, we see that 
$$\alpha(\fra_{p+q})\leq \alpha(\fra_p\cdot\fra_{\frq})\leq \alpha(\fra_p)+\alpha(\fra_q).$$
It is easy to deduce from this (see \cite[Lemma~2.3]{JM}) that
$$\inf_{m\geq 1}\frac{\alpha(\fra_m)}{m}=\lim_{m\to\infty}\frac{\alpha(\fra_m)}{m}.$$
We denote this limit by $\alpha(\fra_{\bullet})$. In particular, we have
$v(\fra_{\bullet})$ when $v\in {\rm Val}_X$, and
$\Arn(\fra_{\bullet})$. We define $\lct(\fra_{\bullet})=\frac{1}{\Arn(\fra_{\bullet})}$
(with the convention that this is infinite if $\Arn(\fra_{\bullet})=0$).  Of course, using the local
Arnold multiplicity, one can define in the same way $\Arn_P(\fra_{\bullet})$
and $\lct_P(\fra_{\bullet})$. 

\begin{example}\label{monomial2}
Suppose that $X=\AAA_{\CC}^n$ and $\fra_{\bullet}$ is a graded sequence of ideals
on $X$, all of them generated by monomials. Using the notation introduced in Example~\ref{monomial}, let $P_m$
denote the Newton polyhedron of $\fra_m$
(see Example~\ref{monomial} for definition). Since $\fra_p\cdot\fra_q\subseteq
\fra_{p+q}$, we have $P_p+P_q\subseteq P_{p+q}$. Let 
$P(\fra_{\bullet})$ be the closure of $\bigcup_m\frac{1}{m}P(\fra_m)$. 

To every $v\in \RR_{\geq 0}^n$ one associates a ``monomial" 
valuation ${\rm val}_v$ of 
$\CC(x_1,\ldots,x_n)$ such that for $f=\sum_{u}c_ux^u\in\CC[x_1,\ldots,x_n]$ we have
$${\rm val}_v(f):=\min\{\langle u,v\rangle\mid c_u\neq 0\}.$$
Note that this is a (multiple of a) divisorial valuation precisely when $v\in\QQ_{\geq 0}^n$. 
Since ${\rm val}_v(\fra_m)=\min_{u\in P_m} \langle u,v\rangle$, we see that
$${\rm val}_v(\fra_{\bullet})=\min_{u\in P(\fra_{\bullet})} \langle u,v\rangle.$$
Furthermore, since $\lct(\fra_m)=\max\{\lambda\geq 0\mid (1,\ldots,1)\in\lambda P_m\}$, it is easy to see that
$$\lct(\fra_{\bullet})=\max\{\lambda\geq 0\mid (1,\ldots,1)\in\lambda P(\fra_{\bullet})\}.$$

Note that $P(\fra_{\bullet})$ is a nonempty convex subset of $\RR_{\geq 0}^n$ with the property that
$P(\fra)+u\subseteq P(\fra_{\bullet})$ for every $u\in\RR_{\geq 0}^n$. Conversely, given $Q\subseteq
\RR_{\geq 0}^n$ that satisfies these properties, we may define
$$\fra_m=(x^u\mid u\in mQ).$$
One can check that $\fra_{\bullet}$ is a graded sequence of ideals such that
$Q=P(\fra_{\bullet})$.
\end{example}

In order to study asymptotic invariants, it is convenient to also consider the associated
sequence of asymptotic multiplier ideals of $\fra_{\bullet}$. Recall that 
these are defined as follows (for details, see \cite[Chapter~11.1]{Lazarsfeld}).
Let $\lambda\in\RR_{\geq 0}$ be fixed.
For every $m, p\geq 1$ we have $\cJ(\fra_m^{\lambda/m})\subseteq
\cJ(\fra_{mp}^{\lambda/mp})$. Indeed, if $h\in\cJ(\fra_m^{\lambda/m})$, then for every divisor 
$E$ over $X$ we have
$$\ord_E(h)>\frac{\lambda}{m}\ord_E(\fra_m)-{\rm Logdisc}(E)
\geq\frac{\lambda}{mp}\ord_E(\fra_{mp})-{\rm Logdisc}(E),$$
hence $h\in\cJ(\fra_{mp}^{\lambda/mp})$. By the Noetherian property, it follows that we have an ideal, denoted $\cJ(\fra_{\bullet}^{\lambda})$, that is equal to 
$\cJ(\fra_{m}^{\lambda/m})$ if $m$ is divisible enough. This is the \emph{asymptotic multiplier 
ideal} of $\fra_{\bullet}$ of exponent $\lambda$. 

For every $p\geq 1$, we put $\frb_p=\cJ(\fra_{\bullet}^p)$, and let 
$\frb_{\bullet}=(\frb_m)_{m\geq 1}$. The following properties are an immediate consequence of the definition:
\begin{enumerate}
\item[i)] If $p<q$, then $\frb_q\subseteq\frb_p$.
\item[ii)] We have $\fra_p\subseteq\frb_p$ for every $p$ (this follows from
$\fra_p\subseteq\cJ(\fra_p)\subseteq\cJ(\fra_{pm}^{1/m})=\cJ(\fra_{\bullet}^p)$ for suitable $m$).
\end{enumerate}
A more subtle property is a consequence of the Subadditivity Theorem
(see \cite[Theorem~11.2.3]{Lazarsfeld}): $\frb_{mp}\subseteq\frb_m^p$ for all $m,p\geq 1$. 

Using these properties one shows that for every valuation $v\in {\rm Val}_X$, we have
$$v(\frb_{\bullet}):=\sup_{m\geq 1}\frac{v(\frb_m)}{m}=\lim_{m\to\infty}\frac{v(\frb_m)}{m},$$
and similarly, 
$$\Arn(\frb_{\bullet}):=\sup_{m\geq 1}\frac{\Arn(\frb_m)}{m}=
\lim_{m\to\infty}\frac{\Arn(\frb_m)}{m}.$$
We also put $\lct(\frb_{\bullet})=1/\Arn(\frb_{\bullet})$.

The basic principle, that was first exploited in \cite{ELMNP}, is that the two sequences
$\fra_{\bullet}$ and $\frb_{\bullet}$ have the same asymptotic invariants. More precisely,
we have the following:

\begin{property}\label{tight}
For every divisor $E$ over $X$, and every $p\geq 1$,
\begin{equation}\label{eq_property_tight}
\ord_E(\fra_{\bullet})-\frac{{\rm Logdisc}(E)}{p}<\frac{\ord_E(\frb_p)}{p}\leq
\frac{\ord_E(\fra_p)}{p}.
\end{equation}
The second inequality follows from $\fra_p\subseteq\frb_p$. For the first one,
let $m$ be divisible enough, so that $\frb_p=\cJ(\fra_{mp}^{1/m})$. 
It follows from the definition of  multiplier ideals that
$$\frac{\ord_E(\frb_p)}{p}=\frac{\ord_E(\cJ(\fra_{mp}^{1/m}))}{p}
>\frac{\ord_E(\fra_{mp})}{pm}-\frac{{\rm Logdisc}(E)}{p}\geq\ord_E(\fra_{\bullet})-\frac{{\rm Logdisc}(E)}{p}.$$

By letting $p$ go to infinity in (\ref{eq_property_tight}), we conclude that
$\ord_E(\fra_{\bullet})=\ord_E(\frb_{\bullet})$. 
\end{property}

\begin{property}\label{equality_for_lct}
One also has $\lct(\fra_{\bullet})=\lct(\frb_{\bullet})$. For this, see \cite[Proposition~2.13]{JM}.
\end{property}

As a consequence, one gets a formula describing the asymptotic log canonical threshold
in terms of asymptotic orders of vanishing, just as for one ideal.

\begin{property}\label{description_asympt_lct}
For every graded sequence of ideals $\fra_{\bullet}$, we have
\begin{equation}\label{eq_description_asympt_lct}
\lct(\fra_{\bullet})=\inf_E\frac{{\rm Logdisc}(E)}{\ord_E(\fra_{\bullet})},
\end{equation}
where the infimum is over all divisors $E$ over $X$. The inequality ``$\leq$"
follows from $\lct(\fra_m)\leq\frac{{\rm Logdisc}(E)}{\ord_E(\fra_m)}$ by multiplying by $m$, and letting $m$
go to infinity. For the reverse inequality, given $m\geq 1$, there is a divisor $F_m$
over $X$ such that $\lct(\frb_m)=\frac{{\rm Logdisc}(E_m)}{\ord_{E_m}(\frb_m)}$. Using Property~\ref{tight}, we deduce
$$m\cdot\lct(\frb_m)=\frac{{\rm Logdisc}(E_m)}{\ord_{E_m}(\frb_m)/m}\geq
\frac{{\rm Logdisc}(E_m)}{\ord_{E_m}(\frb_{\bullet})}=
\frac{{\rm Logdisc}(E_m)}{\ord_{E_m}(\fra_{\bullet})}\geq\inf_E\frac{{\rm Logdisc}(E)}{\ord_{E}(\fra_{\bullet})}.$$
Letting $m$ go to infinity, and using Property~\ref{equality_for_lct}, we get the inequality
 ``$\geq$" 
in (\ref{eq_description_asympt_lct}).
\end{property}

\subsection{A question about asymptotic log canonical thresholds}
As we will see in Example~\ref{example_monomial_conjecture} below, the infimum in
(\ref{eq_description_asympt_lct}) is not, in general, a minimum. In order to have a chance to
get a valuation that realizes that infimum, we need to enlarge the class of valuations we consider.

A \emph{quasi-monomial} valuation $v$ of the function field of $X$ is a valuation 
$v\in {\rm Val}_X$ that is monomial
in a suitable system of coordinates on a model over $X$. More precisely, there is a projective,
birational morphism $\pi\colon Y\to X$, with $Y$ nonsingular, 
$\alpha=(\alpha_1,\ldots,\alpha_n)\in
\RR_{\geq 0}^n$, and local coordinates
$y_1,\ldots,y_n$ at a point $P\in Y$ such that if $f\in\cO_{Y,P}$ is written as
$f=\sum_{\beta\in\ZZ_{\geq 0}^n}c_{\beta}y^{\beta}$ in $\widehat{\cO_{Y,P}}$, then
$$v(f)=\min\{\langle\alpha,\beta\rangle \mid c_{\beta}\neq 0\}.$$
If $E_i\subset Y$ is the divisor defined at $P$ by $(y_i)$, we put
$${\rm Logdisc}(v):=\sum_{i=1}^n\alpha_i\cdot {\rm Logdisc}(E_i).$$
One can show that this definition is independent of the model $Y$ we have chosen.
We refer to \cite[\S 3]{JM} for this and for  other basic facts about quasi-monomial
valuations, and in particular, for their description as Abhyankar valuations.
Note that if $E$ is a divisor over $X$, and $\alpha$ is a non-negative real number, then
$\alpha\cdot\ord_E$ is a quasi-monomial valuation, and ${\rm Logdisc}(\alpha\cdot\ord_E)=
\alpha\cdot {\rm Logdisc}(E)$. 

It is easy to see that if $v$ is a quasi-monomial valuation and $\fra_{\bullet}$
is a graded sequence of ideals, we still have $\lct(\fra_{\bullet})
\leq\frac{{\rm Logdisc}(v)}{v(\fra_{\bullet})}$.
The following conjecture was made in \cite{JM} (in a somewhat more general form).
It says that the asymptotic log canonical threshold of a graded sequence of ideals
can be computed by a quasi-monomial valuation.

\begin{conjecture}\label{conj_graded_sequence}
If $\fra_{\bullet}$ is a graded sequence of ideals on $X$, then there is a quasi-monomial
valuation $v$ of the function field of $X$ such that
$$\lct(\fra_{\bullet})=\frac{{\rm Logdisc}(v)}{v(\fra_{\bullet})}.$$
\end{conjecture}

Note that the conjecture is trivially true if $\lct(\fra_{\bullet})=\infty$. Indeed, in this case any
valuation $v$ such that $v(\fra_{\bullet})=0$ satisfies the required condition (for example,
we can take $v=\ord_E$, where $E$ has center at some point not contained in
$V(\fra_m)$, where $m$ is such that $\fra_m$ is nonzero).

\begin{example}\label{example_monomial_conjecture} 
Suppose that $\fra_{\bullet}$ is a graded sequence of ideals in 
$\CC[x_1,\ldots,x_n]$, all of them generated by monomials. As in Example~\ref{monomial2}, we put $P_m=P(\fra_m)$, and 
let $P(\fra_{\bullet})$ be the closure of $\bigcup_m\frac{1}{m}P_m$. 
We put $e=(1,\ldots,1)\in\RR^n$.
It follows from
Example~\ref{monomial2} that 
if $\lct(\fra_{\bullet})<\infty$, then $\Arn(\fra_{\bullet})\cdot e$ lies on the boundary of the
convex set $P(\fra_{\bullet})$. In this case, there is a nonzero affine linear function
$h$ such that $P(\fra_{\bullet})\subseteq\{u\mid h(u)\geq 0\}$ and 
$h(\Arn(\fra_{\bullet})\cdot e)=0$ (see, for example, \cite[Theorem~4.3]{Bro}).
If $h(x_1,\ldots,x_n)=\alpha_1x_1+\ldots+\alpha_nx_n+b$, then it is easy to see
that $(\alpha_1,\ldots,\alpha_n)\in\RR_{\geq 0}^n\smallsetminus\{0\}$, and the 
``monomial" valuation
$w_{\alpha}$ satisfies $\lct(\fra_{\bullet})=\frac{{\rm Logdisc}(w_{\alpha})}{w_{\alpha}(\fra_{\bullet})}$.

On the other hand, it is easy to construct examples of such sequences $\fra_{\bullet}$
for which there is no divisor $E$ over $X$ such that $\lct(\fra_{\bullet})=
\frac{{\rm Logdisc}(E)}{\ord_E(\fra_{\bullet})}$. Indeed, suppose that 
$$Q=\{(u_1,u_2)\in\RR_{\geq 0}^2\mid (u_1+1)u_2\geq 1\}.$$ 
As in Example~\ref{monomial2}, we take $\fra_m=(x^ay^b\mid (a,b)\in mQ)$, so that
$P(\fra_{\bullet})=Q$. In particular, we get $\Arn(\fra_{\bullet})=\eta=\frac{-1+\sqrt{5}}{2}$.
One can show that since all $\fra_m$ are generated by monomials, every $E$ with 
$\lct(\fra_{\bullet})=
\frac{{\rm Logdisc}(E)}{\ord_E(\fra_{\bullet})}$ is a toric divisor (see \cite[Proposition~8.1]{JM}).
In this case, if $\ord_E(x)=\alpha$ and $\ord_E(y)=\beta$, then
$$\ord_E(\fra_{\bullet})=
\min\{u_1a+u_2b\mid (u_1,u_2)\in\RR_{\geq 0}^2, (u_1+1)u_2\geq 1\}.$$
One deduces $\ord_E(\fra_{\bullet})=2\sqrt{\alpha\beta}-\alpha$, and a simple computation
implies $\alpha/\beta=1-\eta\not\in\QQ$, a contradiction.
\end{example}

The space ${\rm Val}_X$  has a natural topology.
This is the weakest topology that makes all maps ${\rm Val}_X\ni v\to v(f)\in\RR_{\geq 0}$
continuous, where $f\in R$. One can extend the log discrepancy map from quasi-monomial
valuations
to get a lower-semicontinuous function ${\rm Logdisc}\colon {\rm Val}_X\to\RR_{\geq 0}$,
such that ${\rm Logdisc}(v)>0$ if $v$ is nontrivial\footnote{The trivial valuation is the one that takes value zero on every nonzero element of the fraction field of $R$.}. The rough idea
is to approximate each nontrivial valuation by quasi-monomial valuations, and to take the supremum
of the log discrepancies of these valuations. See
\cite[\S 5]{JM} for the precise definition, which is a bit technical.

The following is one of the main results in \cite{JM}.

\begin{theorem}
If $\fra_{\bullet}$ is a graded sequence of ideals on $X$, then there is a 
valuation $v\in {\rm Val}_X$ such that
\begin{equation}\label{eq_last}
\lct(\fra_{\bullet})=\frac{{\rm Logdisc}(v)}{v(\fra_{\bullet})}.
\end{equation}
\end{theorem}

We expect that every valuation as in the above theorem has to be quasi-monomial
(in particular, this would give a positive answer to Conjecture~\ref{conj_graded_sequence}).

\begin{conjecture}\label{conj2}
If $\fra_{\bullet}$ is a graded sequence of ideals on $X$ with $\lct(\fra_{\bullet})<\infty$,
and if $v\in {\rm Val}_X$ is a nontrivial valuation such that (\ref{eq_last}) holds, then
$v$ is a quasi-monomial valuation.
\end{conjecture}

\begin{theorem} ${\rm (}$\cite{JM}${\rm )}$
Conjecture~\ref{conj2} holds when ${\rm dim}(X)=2$. 
\end{theorem}

In the above discussion we only considered the asymptotic invariant $\lct(\fra_{\bullet})$,
constructed from the log canonical threshold. 
One can consider also asymptotic versions constructed from the higher
jumping numbers of multiplier ideals, as follows. If $\frq$ is a fixed nonzero ideal
on $X$, and if $\fra$ is a proper ideal, then
$$\lct^{\frq}(\fra):=\min\{\lambda\in\RR_{\geq 0}\mid \frq\not\subseteq\cJ(\fra^{\lambda})\}.$$
When $\fra$ is fixed and we let $\frq$ vary, we obtain in this way all the jumping numbers
of $\fra$. 
If $\fra_{\bullet}$ is a graded sequence of ideals of $X$, one defines
$$\lct^{\frq}(\fra_{\bullet}):=\sup_m m\cdot\lct^{\frq}(\fra_m)=
\lim_{m\to\infty}m\cdot\lct^{\frq}(\fra_m).$$
The results in this section work if we replace $\lct(\fra_{\bullet})$ by
$\lct^{\frq}(\fra_{\bullet})$, and the conjectures also make sense in this more general setting. 

For technical reasons, as well as for possible applications in the analytic setting (see below),
it is convenient to work in a more general setting, when $X$ is an excellent scheme. It is shown in \cite{JM} that the above results on asymptotic invariants
also hold in this setting, and furthermore, in order to prove the above conjectures in the general setting, it is enough to prove them when $X=\AAA_{\CC}^n$.

One can interpret Conjecture~\ref{conj_graded_sequence} as predicting a finiteness
statement for arbitrary graded sequences of ideals. One can consider it as an algebraic analogue of the Openness Conjecture of Demailly and Koll\'{a}r \cite{DK}. Let us briefly recall this conjecture. Suppose that $\phi$ is a psh (short for plurisubharmonic) function\footnote{We do not give the precise definition, but recall that to an ideal $\fra$ of regular (or holomorphic) functions
on $U\subseteq\CC$, generated by  
$f_1,\ldots,f_r$, one associates a psh function $\phi_{\fra}(z)=
\frac{1}{2}{\rm log}(\sum_{i=1}^n|f_i(z)|^2)$. More interesting examples are obtained by taking suitable limits of such functions.} on an open subset
$U\subseteq\CC$. The \emph{complex singularity exponent} of $\phi$ at $P$
is 
$$c_P(\phi):=\sup\{s>0\mid {\rm exp}(-2s\phi)\,\text{is locally integrable around}\,P\}$$
(compare with the analytic definition of the log canonical threshold in the case when 
$\phi=\phi_{\fra}$). The Openness Conjecture asserts that
the set of those $s>0$ such that ${\rm exp}(-2s\phi)$ is integrable around $P$ is open;
in other words, that ${\rm exp}(-2c_P(\phi)\phi)$ is not integrable around $P$. 
In the case when $\phi=\phi_{\fra}$ for an ideal $\fra$ of regular (or holomorphic) functions,
then this assertion can be proved using resolution of singularities, in the same way that
we proved Theorem~\ref{analytic_description}.

There is no graded sequence of ideals associated to a psh function. However, one can associate to such a function a sequence of ideals $\frb_{\bullet}$ of holomorphic functions that behaves in a similar fashion with the sequence of asymptotic multiplier ideals of a graded sequence of ideals
(see \cite{DK} for a description of this construction). 
We hope that using this formalism, one can show that a positive answer to
Conjecture~\ref{conj_graded_sequence} would imply the Openness Conjecture.

\providecommand{\bysame}{\leavevmode \hbox \o3em
{\hrulefill}\thinspace}

\end{document}